\documentclass{amsart}
\usepackage{amsfonts,amssymb,latexsym}
\usepackage{amsmath,amsthm}
\usepackage{eucal}
\usepackage[notcite,notref]{showkeys}
\usepackage[latin1]{inputenc}
\usepackage{color}

\newtheorem{theorem}{Theorem}[section]
\newtheorem{lemma}[theorem]{Lemma}
\newtheorem{proposition}[theorem]{Proposition}
\newtheorem{corollary}{Corollary}[section]
\newtheorem{hypothesis}{Hypothesis}
\theoremstyle{definition}
\newtheorem{definition}{Definition}[section]

\theoremstyle{remark}
\newtheorem{remark}{Remark}[section]
\def\R{{\mathbb R}}
\def\N{{\mathbb N}}

\def\supp{\mathop{\rm supp}\nolimits}

\def\M{{\mathcal M}}
\newcommand{\dist}[2]{\Bigl\langle #1, #2 \Bigr\rangle}

\numberwithin{equation}{section}
\newcommand{\tendsto}[1]{\renewcommand{\arraystretch}{0.5}
\begin{array}[t]{c}
\longrightarrow \\
{ \scriptstyle #1 }
\end{array}
\renewcommand{\arraystretch}{1}}

\newcommand{\weaktendsto}[1]{\renewcommand{\arraystretch}{0.5}
\begin{array}[t]{c}
\rightharpoonup \\
{ \scriptstyle #1 }
\end{array}
\renewcommand{\arraystretch}{1}}

\begin{document}
\title[Asymptotic stability for  non positive perturbations of the peakon]{Asymptotic stability for some non positive perturbations of the Camassa-Holm peakon with application to the antipeakon-peakon profile}


\author[L. Molinet]{Luc Molinet}

\address{Luc Molinet, Institut Denis Poisson, Universit\'e de Tours, Universit\'e d'Orl\'eans, CNRS,  Parc Grandmont, 37200 Tours, France.}
\email{Luc.Molinet@univ-tours.fr}

\date{\today}

\begin{abstract}
We continue our investigation on the asymptotic stability of the peakon . In a first step we extend our asymptotic stability result  \cite{L} in the class of functions  whose negative part of the momentum density  is supported in  $]-\infty,x_0] $ and the positive part in $[x_0,+\infty[$  for some $ x_0\in \R$. In a second step this enables us to prove the asymptotic stability of well-ordered train of antipeakons-peakons and, in particular, of the antipeakon-peakon profile. Finally, in the appendix we prove that in the case of a non negative momentum density the energy at the left of any given point decays to zero as time goes to $+\infty $.  This leads to an improvement of the asymptotic stability result stated  in  \cite{L}.
 \end{abstract}

\maketitle

\section{Introduction}
\noindent 
In this paper we continue our investigation of the asymptotic stability of the peakon for the Camassa-Holm  equation (C-H) by studying a particular case of solutions with a non signed momentum density. This leads to an asymptotic stability result for the antipeakon-peakon profile with respect to  some perturbations. 

Recall that the  Camassa-Holm equation reads 
\begin{equation}
u_t -u_{txx}=- 3 u u_x +2 u_x u_{xx} + u u_{xxx}, \quad
(t,x)\in\R^2, \label{CHCH}\\
\end{equation}
and can be derived as a model for the propagation of unidirectional
shalow water waves over a flat bottom   (\cite{CH1}, \cite{Johnson}). 
A rigorous derivation  of the Camassa-Holm equation from the full water waves problem is obtained in    \cite{AL} and  \cite{CL}.

(C-H) is completely integrable (see \cite{CH1},\cite{CH2},
\cite{C1} and \cite{CGI}) and enjoys also a geometrical derivation (cf. \cite{K1}, \cite{K2}). It possesses among  others the
following invariants
\begin{equation} 
M(v)= \int_{\R} (v-v_{xx}) \, dx , \; E(v)=\int_{\R} v^2(x)+v^2_x(x)
\, dx \mbox{ and } F(v)=\int_{\R} v^3(x)+v(x)v^2_x(x) \, dx\;
\label{E}
\end{equation}
and can be written in Hamiltonian form as
\begin{equation}
\partial_t E'(u) =-\partial_x F'(u) \quad .
\end{equation}
It is also worth noticing that \eqref{CH} can be rewritted as
\begin{equation}\label{CHy}
y_t +u y_x+2 u_x y =0 
\end{equation}
 that is a transport equation for the momentum density  $y=u-u_{xx} $.

Camassa and Holm \cite{CH1} exhibited peaked solitary waves
solutions to (C-H) that  are
 given by
$$ u(t,x)=\varphi_c(x-ct)=c\varphi(x-ct)=ce^{-|x-ct|},\; c\in\R.$$
They are called peakon whenever $ c>0 $ and antipeakon  whenever
$c<0$, the profile $ \varphi_c $ being  the unique $H^1$- weak solution of  the differential equation
\begin{equation}\label{eqpeakon}
-c\varphi_c +c \varphi_c''+\frac{3}{2} \varphi_c^2= \varphi_c \varphi_c'' + \frac{1}{2}( \varphi_c')^2
\; .
\end{equation}
Note that the initial value problem associated with (C-H) has to be rewriten as
\begin{equation}
\left\{ \begin{array}{l}
u_t +u u_x +(1-\partial_x^2)^{-1}\partial_x (u^2+u_x^2/2)=0\\
\label{CH} \; 
u(0)=u_0, 
\end{array}
\right.
\end{equation}
to give a meaning to these solutions.

Their stability seems not
to enter the general framework developed
 for instance in \cite{Benjamin}, \cite{GSS}, especially because of the non smoothness of the peakon. However, Constantin and
Strauss \cite{CS1} succeeded in proving their orbital stability by
a direct approach.  
 
 In \cite{L}, making use of the finite speed propagation of the momentum density, we proved a rigidity result for solutions with a non negative momentum density. 
Following the framework developed by Martel and Merle  (see for instance \cite{MM1}, \cite{MM2}) this leads to 
 an asymptotic stability result for the peakon with respect to non negative perturbations of the  momentum density.

 In this paper we continue our study of the asymptotic stability of the peakons of the Camassa-Holm equation. In a first part we extend our result in \cite{L} 
  on the asymptotic stability result for non negative initial momentum density $ y_0\in {\mathcal M}_+ $  by considering the case where $ \supp y_0^- \subset ]-\infty,x_0] $ and 
  $ \supp y_0^+\subset [x_0,+\infty[ $ for some $ x_0\in \R $. Here $ y_0^- $ and $ y_0^+ $ denote respectively the negative and positive part of $ y_0 $. 
  In a second part we  use this last result to prove an asymptotic stability result for well-ordered train of antipeakons and peakons.

   It is well-known (cf. \cite{E}) that under the above  hypothesis on $ y_0$, the solution $ u $ exists for all positive times in $ Y $ and that there exists a $ C^1$-function
   $ t\mapsto x_0(t) $ such that for all non negative times, $ \supp y^-(t) \subset ]-\infty,x_0(t)] $ and 
  $ \supp y^+(t)\subset [x_0(t),+\infty[ $. On the other hand, in contrast to the case of a non negative momentum density, no uniform in time bound on the total variation of $ y $ is known in this case. Actually, only an exponential estimate on $ \|y(t)\|_{\mathcal M}$ has been derived (see \eqref{yL1}).   The main novelty of this work is the proof that,  also in this case, the constructed  asymptotic object has a non negative  momentum density and is  $ Y$-almost localized.  
  To prove our result we separate two possible behaviors of $ x_0(t) $.The first case  corresponds to the case where the point $ x_0(t) $ travels far to the left of the bump (at least for large times). Then locally around the bump, the momentum density is non negative and we can argue as in the case of a non negative momentum density.  The boundedness of the total variation of $ y $  around the bump being  proved by using 
  an almost  decay result of $ E(\cdot) +M(\cdot) $ at the right of a point that travels to the right between $ x_0(\cdot) $ and the bump but far away from these last ones.
  The second case is more involved and corresponds to the case where $ x_0(t) $ stays always close to the bump. Then we  prove that the total variation of $ y $ has to decay exponentially fast in time  in a small interval at the left of $x_0(\cdot)$. This enables us to prove an uniform in time estimate on the total variation of $ y^+(t) $ and thus of $ y(t) $ by using the conservation of  $M(u)$ along the trajectory. We can thus pass to the limit on a sequence of times. The fact that the asymptotic object has a non negative momentum  density follows by extending the   decay of the total variation of   $y(t)   $ to an interval that grows with time.
 
At this stage we would like to emphasize that we are not able to consider  antipeakon-peakon  collisions. Actually, our hypothesis forces the antipeakons to be initially  at the left of the peakons and this property is preserved for positive times. Recently, Bressan and Constantin (\cite{BC1}, \cite{BC2}) succeed to  construct global conservative and dissipative solutions of the \eqref{CH} for initial data in $ H^1(\R) $ by using scalar conservation laws techniques (see also \cite{HR1}, \cite{HR4} for a slightly different point of view). These class of solutions take into account the antipeakon-peakon  collisions which are studied in more details in  \cite{HR2}  and \cite{HR3} (see also \cite{GH} and references therein).
\subsection{Statement of the results}
Before stating our results let us introduce the function space where our initial data will take place. Following \cite{CM1}, we introduce the following space of functions
\begin{equation}
Y=\{ u\in H^1(\R)    \mbox{ such that  }  u-u_{xx} \in {\mathcal M}(\R)  \} \; .
\end{equation}
We denote by $ Y_+ $ the closed subset of $ Y $ defined by  $Y_+=\{u\in Y \, /\, u-u_{xx}\in \M_+ \} $ where $ {\mathcal M}_+ $ is the set of non negative finite Radon measures on $ \R$.

Let $ C_b(\R)$  be the set of bounded continuous functions on $ \R $, $ C_0(\R)$  be  the set of continuous functions  on $ \R $ that tends to $ 0 $ at infinity and let $ I \subset \R$ be an interval.  
 A sequence $\{\nu_n\}\subset {\mathcal M} $ is said to converge tightly (resp. weakly) towards $ \nu\in {\mathcal M} $ if for any $ \phi\in C_b(\R) $ (resp. $C_0(\R)$), $ \langle \nu_n,\phi\rangle \to
  \langle \nu,\phi\rangle $. We will then write $ \nu_n  \rightharpoonup \! \ast \; \nu $ tightly  in $ \M $ (resp. $ \nu_n  \rightharpoonup \! \ast \; \nu $ in $\M$).
 
 Throughout this paper, $ y\in C_{ti}(I;\M) $ (resp.   $ y\in C_{w}(I;\M) $) will signify that for any $ \phi\in C_b(\R) $
  (resp. $\phi\in C_0(\R)$) , 
 $ t\mapsto \dist{y(t)}{\phi} $ is continuous on $ I$ and $ y_n \rightharpoonup \! \ast \; y $ in $ C_{ti}(I;\M) $ (resp. $ y_n \rightharpoonup \! \ast \; y $ in $ C_{w}(I;\M) $) will signify that for any $ \phi\in C_b(\R) $ (resp. $C_0(\R)$),
 $ \dist{y_n(\cdot)}{\phi}\to \dist{y(\cdot)}{\phi} $ in $ C(I)$.

As explained above, the aim of this paper is to extend the results in \cite{L} under  the following hypothesis.
\begin{hypothesis}\label{hyp}
We will say that $ u_0\in Y $ satisfies  Hypothesis 1 if  there exists $ x_0\in \R $ such that its momentum density $ y_0 =u_0-u_{0,xx} $ satisfies  
\begin{equation}\label{hyp1}
\supp y_0^{-}\subset ]-\infty,x_0] \quad \text{ and } \quad \supp y_0^+ \subset [x_0,+\infty[ \; .
\end{equation}
\end{hypothesis}
 As in \cite{L}, the key  tool to prove our results is the following rigidity property for $ Y$-almost localized solutions of \eqref{CH} with non negative density momentum :
\begin{theorem}[\cite{L}]\label{liouville}
 Let $ u \in C(\R; H^1(\R)) $,  with $ u-u_{xx}\in C_{w}(\R; \M_+) $, be a $ Y$-almost localized solution of \eqref{CH} that is not identically vanishing. Then there exists $ c^*	>0 $ and $ x_0\in \R $ such that  $$
 u(t)=c^* \, \varphi(\cdot -x_0-c^* t) , \quad \forall t\in \R \, .
 $$
\end{theorem}
Recall that a $ Y$-almost localized solution is defined in the following way :
\begin{definition}\label{defYlocalized}
 We say that  a solution $ u \in C(\R; H^1(\R)) $ with $ u-u_{xx}\in C_{w}(\R; \M_+) $ of \eqref{CH} is $ Y$-almost localized if there exist $ c>0 $ and a $ C^1 $-function $ x(\cdot) $, with $ x_t\ge c>0 $, for which for any $ \varepsilon>0 $, there exists $ R_{\varepsilon}>0 $ such that for all $ t\in\R $ and all $ \Phi\in C(\R) $ with $0\le \Phi\le 1 $ and $ \supp \Phi \subset [-R_\varepsilon,R_\varepsilon]^c $.
  \begin{equation}\label{defloc}
 \int_{\R} (u^2(t)+u_x^2(t))  \Phi(\cdot-x(t)) \, dx + \Bigl\langle  \Phi(\cdot-x(t)), u(t)-u_{xx}(t)\Bigr\rangle \le \varepsilon \; .
\end{equation}
 \end{definition}
Our first new result is the asymptotic stability of the peakon with respect to perturbations that satisfy Hypothesis \ref{hyp} :
\begin{theorem} \label{asympstab} 
Let $ c>0 $ be fixed. There exists an  universal constant $0<\eta\le 2^{-10}  $ such that for any $ 0<\theta<c $ and any $ u_0\in Y $, satisfying Hypothesis 1., such that 
\begin{equation}\label{difini}
\|u_0-\varphi_c \|_{H^1} \le \eta \Bigl(\frac{\theta}{c}\Bigr)^{8}\; ,
\end{equation}
there exists $ c^*>0 $ with $ |c-c^*|\ll c $ and a $C^1$-function $ x \, : \, \R \to \R $ 
 with $ \displaystyle\lim_{t\to \infty} \dot{x}(t)=c^* $  such that
\begin{equation}
u(t,\cdot+x(t)) \weaktendsto{t\to +\infty} \varphi_{c^*} \mbox{ in } H^1(\R) \; ,
\end{equation}
where $ u\in C(\R; H^1) $ is the solution emanating from $ u_0 $.
Moreover, 
\begin{equation}\label{cvforte}
\lim_{t\to +\infty} \|u(t) -\varphi_{c^*}(\cdot-x(t))\|_{H^1(|x|>\theta t )}=0 \; .
\end{equation}
\end{theorem}
\begin{remark}
We emphasize that \eqref{cvforte} gives a strong $ H^1 $-convergence result at the left of the line $ x=-\theta t $. This is due to the fact that, contrary to KdV type equation, there is no small linear waves travelling to the left for the Camassa-Holm equation. In the appendix we even show that in the case of a non negative momentum density, all the energy is traveling to the right.
\end{remark}
Combining this asymptotic stability result with the one obtained in \cite{L}  and the orbital stability of well ordered trains of antipeakons-peakons proven in \cite{EL3}, we are able to prove the asymptotic stability of such trains that contains, as a particular case, the asymptotic stability of the antipeakon-peakon profile.

\begin{theorem}\label{asympt-mult-peaks}
Let be given $ N_- $ negative   velocities $ c_{-N_-} <..<c_{-2}<c_{-1}<0 $, $ N_+ $ positive velocities $0<c_1<c_2<..<c_{N_+} $ and $ 0<\theta_0<\min(|c_{-1}|,c_1)/4 $.
There exist   $ L_0>0 $
 and $ \varepsilon_0>0 $ such that if  the initial data  $ u_0\in Y $ satisfies Hypothesis 1 with 
 \begin{equation}
 \|u_0-\sum_{j=1}^{N_-} \varphi_{c_{-j}}(\cdot-z_{-j}^0)-\sum_{j=1}^{N_+} \varphi_{c_j}(\cdot-z_j^0) \|_{H^1} \le \varepsilon_0^2 ,\label{inini}
 \end{equation}
 for some 
 $$ z^0_{N_-}<\cdot\cdot<z^0_{-1}<z^0_1<\cdot\cdot<z^0_{N_+} \text{ with } |z^0_j-z^0_q| \ge L_0 \text{ for }   j\neq q \; ,
 $$
 then there exist $c_{-N_-}^*<..<c_{-2}^*<c_{-1}^*<0< c_1^*<..<c_{N_+}^*  $ 
  with $ |c_j^*-c_j|\ll |c_j| $ and $ C^1$-functions $t\mapsto x_j(t), $  with  $ \dot{x}_j(t) \to c_j^* $ as $ t\to +\infty $, 
 $ j\in [[-N_-,N_+]]/\{0\} $,  such that the solution $ u \in C(\R_+;H^1(\R)) $ of \eqref{CH} emanating from $ u_0 $ satisfies 
\begin{equation}\label{mul1}
u(\cdot+x_j(t)) \weaktendsto{t\to +\infty} \varphi_{c_j^*}  \mbox{ in } H^1(\R), \; \forall j\in [[-N_-,N_+]]/\{0\} \; .
\end{equation}

 Moreover, 
\begin{equation}\label{mul2}
u-\sum_{j=1}^{N_-} \varphi_{c_{-j}^*}(\cdot-x_{-j}(t))-\sum_{j=1}^{N_+} \varphi_{c_j^*}(\cdot-x_j(t)) \tendsto{t\to +\infty} 0 \mbox{ in } H^1\Bigl(|x|>\theta_0 t\Bigr)\; .
\end{equation}
 \end{theorem}
This paper is organized as follows : in the next section we recall the well-posedness results for the class of solutions we will work with. In Section 3, 
we derive an almost monotonicity result that is a straightforward adaptation  of the one proven in \cite{L}. Section 4 is devoted to the proof of the $ Y$-almost localization  of the elements of the $ \omega$-limit set of  the orbits associated with initial data  satisfying Hypothesis 1 and being $ H^1$-close enough to a peakon. This is the main new contribution of this paper. 
In Section 5 we deduce  the asymptotic stability results. Finally  in the appendix we present  an improvement of the asymptotic  stability result given in \cite{L} by noticing that all the energy of any  solution to \eqref{CH} with a non negative density momentum is traveling to the right. Moreover, for such a solution, the energy at the the left of any given point  decays to zero 	as $ t \to +\infty $. 
\section{Global well-posedness results} 
We first recall some obvious estimates that will be useful in the sequel of this paper.  
Noticing that $ p(x)=\frac{1}{2}e^{-|x|} $ satisfies $p\ast y=(1-\partial_x^2)^{-1} y $ for any $y\in H^{-1}(\R) $ we easily get
$$
\|u\|_{W^{1,1}}=\|p\ast (u-u_{xx}) \|_{W^{1,1}}\lesssim \| u-u_{xx}\|_{\M}
$$
and 
$$
\|u_{xx}\|_{\M}\le \|u\|_{L^1}+ \|u-u_{xx} \|_{\M} 
$$
 which ensures  that 
\begin{equation} \label{bv}
Y\hookrightarrow  \{ u\in W^{1,1}(\R) \mbox{ with } u_x\in {\mathcal BV}(\R) \} \; .
\end{equation}
 It is also worth noticing that  for $ v\in C^\infty_0(\R) $, satisfying Hypothesis \ref{hyp},
\begin{equation}\label{formulev}
v(x)=\frac{1}{2} \int_{-\infty}^x e^{x'-x} (v-v_{xx})(x') dx' +\frac{1}{2} \int_x^{+\infty} e^{x-x'} (v-v_{xx})(x') dx'
\end{equation}
and 
$$
v_x(x)=-\frac{1}{2}\int_{-\infty}^x e^{x'-x} (v-v_{xx})(x') dx' + \frac{1}{2}\int_x^{+\infty} e^{x-x'} (v-v_{xx})(x') dx' \; ,
$$
so that for $ x\le x_0 $ we get 
$$
v_x(x) =v(x)-e^{-x} \int_{-\infty}^x e^{x'} y(x') \, dx' \ge v(x)
$$
whereas for $ x\ge x_0 $ we get 
$$
v_x(x) =-v(x)+e^{x} \int^{+\infty}_x e^{-x'} y(x') \, dx' \ge -v(x)
$$
Throughout this paper, we will denote $ \{\rho_n\}_{n\ge 1} $ the mollifiers defined by 
\begin{equation} \label{rho}
\rho_n=\Bigl(\int_{\R} \rho(\xi) \, d\xi 
\Bigr)^{-1} n \rho(n\cdot ) \mbox{ with } \rho(x)=\left\{ 
 \begin{array}{lcl} e^{1/(x^2-1)} & \mbox{for} & |x|<1 \\
0 & \mbox{for} & |x|\ge 1 
\end{array}
\right.
\end{equation}
Following \cite{E} we approximate $ v\in Y $ satisfying  Hypothesis \ref{hyp} by the  sequence of functions
\begin{equation}\label{app}
v_n=  p\ast y_n \;\text{with}\;y_n=-(\rho_n\ast y^-)(\cdot+\frac{1}{n})+(\rho_n\ast y^+)(\cdot-\frac{1}{n})
 \text{ and } y=v-v_{xx}
\end{equation}
that belong to $Y\cap H^\infty(\R) $  and satisfy Hypothesis \ref{hyp} with the same $ x_0 $.
It is not too hard to check that 
\begin{equation}\label{estyn}
\|y_n\|_{L^1} \le \|y\|_{\mathcal M} 
\end{equation}
Moreover,  noticing that 
$$
v_n  = -\Bigl( \rho_n \ast (p \ast y^-)\Bigr) (\cdot+\frac{1}{n})+\Bigl(\rho_n\ast (p\ast  y^+)\Bigr)(\cdot-\frac{1}{n} )
$$
with  $ p\ast y^{\mp} \in H^1(\R) \cap W^{1,1}(\R) $,  we infer that 
\begin{equation}\label{mm}
v_n \to v \in H^1(\R) \cap W^{1,1}(\R) \; .
\end{equation}
that ensures that for any $ v\in Y $ satisfying  Hypothesis \ref{hyp} it holds
\begin{equation}\label{dodo}
 v_x \ge v \text{ on } ]-\infty, x_0[\quad  \text{ and } \quad  v_x\ge -v \text{ on } ]x_0,+\infty[   \; .
\end{equation}
 \begin{proposition}{(Global weak solution \cite{E})}\label{prop1} \\
 Let $ u_0 \in Y $ satisfying Hypothesis \ref{hyp}  for some 
  $ x_0\in\R $.  \vspace*{2mm} \\
 {\bf 1. Uniqueness and global existence :} \eqref{CH} has a unique solution 
 $$
  {u\in C(\R_+;H^1(\R)) \cap C^1(\R_+;L^2(\R))}
  $$
   such that 
  $ y=(1-\partial_x^2)u \in C_{ti}(\R_+; \M) $.  $E(u) $, $ F(u) $ and $ M(u)=\dist{y}{1} $ are conservation laws . Moreover, for any $ t\in \R_+$, the density momentum 
   $ y(t) $ satisfies   $\supp y^{-}(t)\subset ]-\infty,q(t,x_0)] $ and $\supp y^+(t) \subset [q(t,x_0),+\infty[ $ where $ q(\cdot,\cdot)$ is defined by 
 \begin{equation}\label{defq}
  \left\{ 
  \begin{array}{rcll}
  q_t (t,x) & = & u(t,q(t,x))\, &, \; (t,x)\in \R^2\\
  q(0,x) & =& x\, & , \; x\in\R \; 
  \end{array}
  \right. .
  \end{equation} 
  
   \vspace*{2mm} \\
  {\bf 2. Continuity with respect to  the $ H^1$-norm}: For any sequence $ \{u_{0,n}\} $ bounded in $ Y $ that satisfy Hypothesis \ref{hyp} and such that 
    $ u_{0,n} \to u_0 $ in $ H^1(\R ) $,  the emanating sequence of solutions $ \{u_n\} \subset  C^1(\R_+;L^2(\R))\cap C(\R_+;H^1(\R)) $ satisfies for any $ T>0 $
  \begin{equation}\label{cont1}
  u_n \to u \mbox{ in } C([0,T]; H^1(\R)) \; . \end{equation}
 Moreover, if $ \{u_{0,n}\}  $ is the sequence defined by \eqref{app} then 
    \begin{equation}\label{cont2}
 (1-\partial_x^2) u_n  \rightharpoonup  \! \ast \; y \mbox{ in } C_{ti} ([0,T], \M) \; . 
   \end{equation}
       {\bf 3. Continuity with respect to initial data  in $Y$ equipped with its weak topology:}  Let  $ \{u_{0,n}\} $ be a bounded   sequence of  $ Y $ such that 
    $ u_{0,n} \rightharpoonup u_0 $ in $ H^1(\R ) $ and  such that the emanating sequence of solution $ \{u_n\} $ is bounded in
      $  C( [-T_-,T_+];H^1)\cap L^\infty(-T_-,T_+; Y) $  for some $(T_-,T_+)\in (\R_+)^2 $. Then the solution $u$ of \eqref{CH} emanating from $ u_0 $ belongs to $  C( [-T_-,T_+];H^1)$ with $ y=u-u_{xx}\in  C_{w}([-T_-,T_+]; {\mathcal M})$. Moreover, 
       \begin{equation}\label{weakcont}
  u_n \weaktendsto{n\to\infty} u \mbox{ in } C_{w}([-T_-,T_+]; H^1(\R) ) \; ,
  \end{equation}
  and 
     \begin{equation}\label{cont22}
 (1-\partial_x^2) u_n  \rightharpoonup  \! \ast \; y \mbox{ in } C_{w} ([-T_-,T_+], \M) \; . 
   \end{equation}
  \end{proposition}
  \begin{proof}
    For sake of completeness let us recall that the global existence result follows from the following estimate on the density momentum $ y $  of smooth solutions with initial data
   $ u_0 \in Y $  that
   satisfies Hypothesis \ref{hyp} : 
   \begin{eqnarray*}
   \frac{d}{dt}\int_{\R} y_+(t,x)\, dx &=& \frac{d}{dt} \int_{q(t,x_0)}^{+\infty} y(t,x) \, dx \\
    & = &- \int_{q(t,x_0)}^{+\infty} u_x(t,x) y(t,x)  \, dx \\
     & \le & \sup_{x\in\R} (-u_x(t,x)) \int_{q(t,x_0)}^{+\infty} y(t,x) \, dx\\
     & \le &  \sup_{x\in\R} (|u(t,x|)) \int_{\R}  y^+(t,x) \, dx \\
     &\le & \sqrt{E(u_0)}  \int_{\R}  y^+(t,x) \, dx \; ,
   \end{eqnarray*}
  where one used  \eqref{CHy},  \eqref{dodo}, the conservation of the energy and the classical Sobolev inequality :
  \begin{equation}\label{sobo}
  \|v\|_{L^\infty} \le \frac{1}{\sqrt{2}} \|v\|_{H^1} , \quad \forall v\in H^1(\R) \; .
  \end{equation}
  Indeed, then by Gronwall estimate and the conservation of $ M $ one gets 
  \begin{equation}\label{yL1}
  \|y\|_{L^1} \le  2 \exp(\sqrt{E(u_0)}t) \|y_0\|_{L^1} 
  \end{equation}
  which ensures that the associated solution $ u $ can be extended for all positive times. Finally the global existence result for $ u_0\in Y $ follows by  approximating $ u_0 $ as in 
  \eqref{app} and proceeding as in \cite{CM1}.

  In this way, the uniqueness and global existence results are obtained in \cite{E} except the conservation of $ M(u) $ and the  fact that  $y$ belongs to $C_{ti}(\R+; \M)$. 
   In \cite{E}, only the fact that $ y\in L^\infty_{loc}(\R_+,\M) $ is stated. But these properties will follow directly from  \eqref{cont2}  since $ M(u) $ is a conservation law  for any smooth solution  $ u\in C(\R_+;H^3)$ with $ u-u_{xx} \in L^\infty_{loc}(\R_+;L^1(\R)) $. 
  
 To prove \eqref{cont1}, it suffices to notice that, according to the conservation of the $ H^1$-norm   and \eqref{yL1}, the sequence of emanating solution $\{u_n\} $ is bounded in  $ C(\R_+;H^1(\R))\cap L^\infty(]0,T[; W^{1,1}(\R)) $
  with $ \{u_{n,x} \} $ bounded in $ L^\infty(]0,T[; {\mathcal BV}(\R)) $,  for any $ T>0$.  Therefore, there exists $ v\in  L^\infty(\R_+;H^1(\R)) $ with $ (1-\partial_x^2) v \in 
   L^\infty_{loc}(\R_+; {\mathcal M(\R)}) $ such that, for any $ T>0$,
   $$
     u_n \weaktendsto{n\to\infty} v \in L^\infty(]0,T[; H^1(\R))  \mbox{ and }  (1-\partial_x^2) u_n \weaktendsto{n\to\infty} \hspace*{-3mm} \ast \; (1-\partial_x^2) v  \mbox{ in } L^\infty(]-T,T[; {\mathcal M}(\R))  \; .
   $$
  Moreover, in view of \eqref{CH}, $ \{\partial_t u_n\} $ is bounded in $L^\infty(]0,T[; L^2(\R) \cap L^1(\R) )$ and Helly's,  Aubin-Lions compactness and Arzela-Ascoli theorems  ensure that $ v $ is a solution to \eqref{CH} that belongs to $ C_{w}([0,T]; H^1(\R)) $ with  $ v(0)=u_0 $ and that \eqref{cont2} holds. In particular, $ v_t\in L^\infty(]0,T[; L^2(\R)) $ and thus $ v\in C([0,T];L^2(\R)) $. Since $ v\in L^\infty(]0,T[; H^{\frac{3}{2}-} (\R))$, this  actually implies that 
   $ v\in C([0,T];  H^{\frac{3}{2}-}(\R)) $. Therefore, 
 $v$  belongs to the uniqueness class which ensures that $ v=u$. The conservation of $ E(\cdot) $ and the above weak convergence results then lead to \eqref{cont1}.
 
  To prove \eqref{cont2} we use the following  $ W^{1,1} $-Lipschitz  bound that is proven in \cite{CM1} and \cite{E} : Let $ u_{0}^{i} \in Y $ for $ i=1,2 $ and let $u^i \in C([0,T];H^1) \cap L^\infty(0,T;Y) $ be the associated solution of \eqref{CH}. Setting $ M=\sum_{i=1}^2 \|u^{i}-u_{xx}^{i}\|_{L^\infty(0,T;\M)} $, it holds 
 \begin{equation}\label{lip}
 \|u^1-u^2\|_{L^\infty(0,T;W^{1,1})} \lesssim e^{6 M  T } \|u_{0}^1-u_{0}^2\|_{W^{1,1}}\; .
 \end{equation}
 Let $\{u_{0,n}\} $ the sequence defined by \eqref{app}.  The sequence of emanating solutions $\{u_n\} $ 
 is included in $C(\R_+;H^\infty) \cap C(\R_+; W^{1,1}) $. Moreover in view of \eqref{estyn}, \eqref{mm}   and \eqref{lip}, for any $ T>0 $, 
  $\{u_n\} $ is a Cauchy sequence  in $ C([0,T] ; W^{1,1}) $ and thus 
  \begin{equation}\label{convW}
  u_n \to u \in C(\R_+; W^{1,1})  \; .
   \end{equation}
   Now, we notice that for any $ v\in BV(\R) $ and any $ \phi\in C^1(\R) $, it holds 
  $$
  \langle v', \phi\rangle = -\int  v \phi' \; .
  $$
  Therefore, setting $ y_n=u_n-u_{n,xx} $,  \eqref{convW} ensures that, for any $ t\in \R_+ $,
  $$
  \int_{\R} y_n \phi =\int_{\R} (u_n \phi+u_{n,x}  \phi') \to \int_{\R}  (u \phi+u_{x}  \phi')
  = \langle y, \phi \rangle 
  $$
  and thus  $  y_n(t) \rightharpoonup  \! \ast \; y(t) $ tightly in $\M$. Using again that $ \{\partial_t u_{n}\} $ is bounded in 
   $ L^\infty(0,T;L^1) $, Arzela-Ascoli theorem leads then  to \eqref{cont2}.
     Finally \eqref{weakcont}-\eqref{cont22} can be proven exactly in the same way, since   $ u_0\in Y $ and $ \{u_{n}\} $ is bounded in $L^\infty(]-T_-,T_+[; Y) $  by  hypotheses.   
  \end{proof}

   \section{Monotonicity results}
   We have to prove a monotonicity result for our solutions. The novelty with respect to the monotonicity result proven in \cite{L} is that we only require the momentum density to be non negative  at the right of some curve. This is possible since, the differential equation satisfied by $ y $ being local and of order 1, we may test $ y $ with a  test function $ \Phi$ that vanishes on $ \R_-$. Note that this is not possible to use such test function for the energy since the non local term contained in the  time derivative of the energy density imposes to require a condition similar to  \eqref{po} on the test function.
   
As in \cite{MM2}, we introduce the $ C^\infty $-function $ \Psi $ defined  on $ \R $ by 
\begin{equation}\label{defPsi}
\Psi(x) =\frac{2}{\pi} \arctan \Bigl( \exp(x/6)\Bigr) 
\end{equation}
It is easy to check that $ \Psi(-\cdot)=1-\Psi $ on $ \R $, $ \Psi' $ is a  positive even  function and that 
 there exists $C>0 $ such that $ \forall x\le 0 $, 
\begin{equation}\label{psipsi}
|\Psi(x)| + |\Psi'(x)|\le C \exp(x/6) \; .
\end{equation}
Moreover, by direct calculations, it is easy to check that 
\begin{equation}\label{psi3}
|\Psi^{'''}| \le  \frac{1}{2} \Psi' \;\text{on }\R 
\end{equation}
and that 
\begin{equation} \label{po}
\Psi'(x)\ge \Psi'(2)= \frac{1}{3\pi} \frac{e^{1/3}}{1+e^{2/3}} , \quad \forall x\in [0,2] \; .
\end{equation}
We also introduce the function $ \Phi $ defined by 
\begin{equation}\label{defPhi}
\Phi(x) =\left\{ \begin{array}{rcl}
0 &\text{for}& x\le 0\\
x/2  &\text{for}& x\in [0,2] \\
1&\text{for}& x\ge 1
\end{array}
\right.
\end{equation}
\begin{lemma}\label{almostdecay}
  Let $0<\alpha < 1 $ and let $ u\in C(\R;H^1) $ with $ y=(1-\partial_x^2)u \in C_{w}(\R; \M+) $ be a solution of \eqref{CH}, emanating from an initial datum 
  $u_0\in Y$ that satisfies Hypothesis \ref{hyp},  such that there exist $x\,:\, \R_+\to \R $ of class $ C^1 $ with 
   $ \inf_{\R} \dot{x}\ge c_0>0$  and $ R_0>0 $ with
 \begin{equation}\label{loc}
 \|u(t)\|_{L^\infty(|x-x(t)|>R_0)} \le \frac{(1-\alpha) c_0}{2^6} \, , \; \forall t\in\R_+ .
 \end{equation}
  For $ 0<\beta \le \alpha $, $ 0\le \gamma\le \frac{1}{3\pi (1+e^{2/3})} (1-\alpha) c_0 $, $ R>0 $, $ t_0\in\R_+ $  and any $ C^1 $-function 
   $ z\, :\, \R_+ \to \R $ such that 
  \begin{equation}\label{condz}
 \quad (1-\alpha)  \dot{x}(t) \le \dot{z}(t) \le (1-\beta) \dot{x}(t), \quad \forall t\in \R_+,
   \end{equation}
    we set 
     \begin{equation}\label{defI}
 I^{\mp R}_{t_0} (t)=\dist{u^2(t)+u_x^2(t)}{\Psi\Bigl(\cdot- z_{t_0}^{\mp R}(t)\Bigr)}+\gamma \dist{y(t)}{\Phi\Bigl(\cdot- z_{t_0}^{\mp R}(t)\Bigr)},
 \end{equation}
   where 
     \begin{equation}\label{condz2}
z_{t_0}^{\mp R}(t)=x(t_0)\mp R +z(t)-z(t_0)
   \end{equation}
   Let  also $ q(\cdot,\cdot) $ be defined as in  \eqref{defq}.
   Then if 
   \begin{equation}\label{zo1}
    z_{t_0}^{+R}(t) \ge q(t,x_0) ,\text{  for } 0\le t \le t_0  \; , 
    \end{equation}
     it holds 
      \begin{equation}
I^{+R}_{t_0}(t_0)-I^{+R}_{t_0}(t)\le K_0 e^{-R/6} , \quad \forall  0\le t\le t_0 \quad  \label{mono}
\end{equation}
whereas if 
  \begin{equation}\label{zo2}
   z_{t_0}^{-R}(t) \ge q(t,x_0)  ,\text{  for } t\ge t_0 
     \end{equation}
      it holds 
 \begin{equation}
I^{-R}_{t_0}(t)-I^{-R}_{t_0}(t_0)\le K_0 e^{-R/6} , \quad \forall t\ge t_0 \quad , \label{mono2}
\end{equation}
for some constant $ K_0>0 $ that only depends on $ E(u) $, $c_0$, $R_0$ and $ \beta$. 
\end{lemma}
\begin{proof}
The proof is the same as in \cite{L}. The only difference is that we replace $ \Psi $ by $ \Phi $ to test the density momentum $ y$.  

We first approximate $u_0$ by the sequence $ \{u_{0,n}\} \in Y\cap C^\infty(\R) $ as in \eqref{app}. Then, by Proposition \ref{prop1}, the emanating solutions $ u_n $ exist for all positive time  and satisfy \eqref{yL1}. In particular, for any $T_0>0 $ fixed, 
 $$
 \|u_{n,x} \|_{L^\infty(]0,T_0[\times\R)}\le \sup_{t\in ]0,T_0[} \|y_n(t)\|_{L^1(\R)} \le 2\exp(2\sqrt{E(u_0)}T_0) \|y_0\|_{\mathcal M} 
$$
and thus  for any $ \varepsilon>0 $ there exists $ \theta= \theta_{\varepsilon,T_0}>0 $ such that if $ \|u_n-u\|_{L^\infty(]0,T_0[\times\R)}<\theta_{\varepsilon,T_0} $ then 
\begin{equation}
\label{qqq}
|q_n(t,x_0)-q(t,x_0) | <\varepsilon \quad , \forall t\in 0,T_0] \; .
\end{equation}
 Now, for any fixed $ t_0, T>0 $, \eqref{cont1} ensures that there exists $n_0=n_0(t_0+T)\ge 0 $ such that for any $ n\ge n_0 $, 
$$
\|u_n-u\|_{L^\infty(]0,t_0+T[\times \R)} <  \max\Bigl(\frac{\alpha c_0}{2^6},  \theta_{\frac 1 2, t_0+T}\Bigr) \; .
$$
which together with  \eqref{loc} and \eqref{qqq}  force 
\begin{equation}\label{dif}
\sup_{t\in ]0,t_0+T[}  \|u_n\|_{L^\infty(|x-x(t)|>R_0)} <  \frac{(1-\alpha) c_0}{2^5} \end{equation}
and 
\begin{equation}\label{difq}
\ \sup_{t\in ]0,t_0+T[} |q_n(t,x_0-q(t,x_0)| \le \frac 1 2 \; .
\end{equation}
We first prove that    \eqref{mono}  holds on $ [0,t_0] $ with  $ u $ replaced by $ u_n $ for $ n\ge n_0 $. 
The following computations hold for $ u_n $ with $ n\ge n_0$ but , to simplify the notation, we drop the index $ n $. According to  \cite{L} we have
\begin{equation}
\frac{d}{dt}\int_{\R} (u^2+u_x^2) g =\int_{\R} u u_x^2 g'  +2 \int_{\R} u h g' \; , 
\label{go}
\end{equation}
where $ h:=(1-\partial_x^2)^{-1} (u^2+u_x^2/2) $ and  
 \begin{eqnarray}
 \frac{d}{dt}\int_{\R}  y g \, dx 
  &= &  \int_{\R} y u g' +\frac{1}{2} \int_{\R} (u^2-u_x^2) g'  \label{go2} \; .
 \end{eqnarray}
  Applying \eqref{go} with $ g(t,x)=\Psi(x - z^{R}_{t_0}(t)) $ and \eqref{go2} with  $ g(t,x)=\Phi(x - z^{R}_{t_0}(t)) $   we get
  \begin{eqnarray}
 \frac{d}{dt}I^{+R}_{t_0}(t) & = &-\dot{z}(t) \int_{\R}  \Bigl[ \Psi'(u^2+u_x^2) +\gamma \Phi' y \Bigr] +\frac{\gamma}{2} \int_{\R} (u^2-u_x^2) \Phi'\nonumber \\
 &  & + \int_{\R}  \Bigl[ \Psi'  u u_x^2 + \gamma  \Phi' y u \Bigr] +2 \int_{\R} u h \Psi' \nonumber \\
  &\le &   - \dot{z}(t) \int_{\R}  \Bigl[ \Psi'(u^2+u_x^2) +\gamma \Phi' y \Bigr] +\frac{\gamma}{2} \int_{\R} (u^2-u_x^2) \Phi'+J_1+J_2\, . \label{go3}
 \end{eqnarray}
  Now, in view of \eqref{po} and the conditions on $ \gamma$, we have 
 $$
 \dot{z}(t)\Psi' -\frac{\gamma}{2} \Phi' \ge \frac{(1-\alpha)}{4} c_0 \Psi' \quad\text{ on } \R
 $$
 that leads to 
 $$
 -\dot{z}(t) \int_{\R}  \Bigl[ \Psi'(u^2+u_x^2) +\gamma \Phi' y \Bigr]  +\frac{\gamma}{2} \int_{\R} (u^2-u_x^2) \Phi'\le -\frac{(1-\alpha) c_0}{4} \int_{\R} 
  \Bigl[ \Psi'(u^2+u_x^2) +\gamma \Phi' y \Bigr] 
 $$
 where we used that, according to \eqref{condz2} and \eqref{difq}, $ y \ge 0 $ on the support of $ \Phi' $.
  Finally the terms $J_1$ and $ J_2$ are treated as in \cite{EL3}. For instance, 
 to estimate $ J_{1} $ we divide $ \R $ into two regions relating to the size of $ |u| $ as follows
\begin{eqnarray}
J_{1}(t) &= & \int_{|x-x(t)|<R_0} \Bigl[ \Psi' u u_x^2+\gamma \Phi' y u \Bigr]
+ \int_{|x-x(t)|>R_0} \Bigl[ \Psi' u u_x^2+\gamma \Phi' y u \Bigr]\nonumber \\
 & = & J_{11}+J_{12}\quad . \label{J0}
\end{eqnarray}
Observe that \eqref{condz} ensures that $ \dot{x}(t)-\dot{z}(t)\ge \beta c_0 $ for all $t\in \R $  and thus, for $ |x-x(t)|<R_0 $,
 \begin{equation}\label{to1}
  x-z_{t_0}^{R}(t)=x-x(t)-R+(x(t)-z(t))-(x(t_0)-z(t_0))\le  R_0-R-\beta c_0 (t_0-t) 
  \end{equation}
  and thus the decay properties of $ \Psi' $ and the compact support of $ \Phi' $ ensure that 
\begin{eqnarray}
J_{11} (t) &\lesssim & \Bigl[\|u(t)\|_{L^\infty} (\|u_x(t)\|_{L^2}^2+c_0\|y(t)\|_{L^1})\Bigr]  e^{R_0/6}  e^{-R/ 6}
e^{-\frac{\beta}{6} c_0(t_0-t)} \nonumber \\
 & \lesssim  &  \| u_0\|_{H^1}(\|u_0\|_{H^1}^2+c_0\|y_0\|_{L^1}) e^{R_0/6} e^{-R/6}
e^{-\frac{\beta}{6} c_0(t_0-t)} \quad . \label{J11}
\end{eqnarray}
On the other hand,  \eqref{dif}  ensures that for all $ t\in [t_0-T,t_0] $ it holds
\begin{eqnarray}
J_{12} &\le & 4 \| u\|_{L^\infty(|x-x(t)|>R_0)} \int_{|x-x(t)|>R_0}  \Bigl[ \Psi'  u_x^2+\gamma \Phi' y  \Bigr]\nonumber \\
 & \le & \frac{ (1-\alpha) c_0 }{8}  \int_{|x-x(t)|>R_0} \Bigl[ \Psi'  u_x^2+\gamma \Phi' y  \Bigr]\quad .\label{J12}
\end{eqnarray}
Gathering \eqref{J0}, \eqref{J11}, \eqref{J12} and the estimates on $ J_2$  that we omit, since it is exactly the same terms as in \cite{EL3},  we conclude that there exists  $C >0$  only
depending on  $R_0 $ and $ E(u) $ 
  such
that for  $ R \ge R_0 $ and $ t\in [0,t_0] $ it holds 
\begin{equation}
\frac{d}{dt} I^{+R}_{t_0}(t) \le -\frac{ (1-\alpha)
c_0}{8}\int_{\R} \Bigl[ \Psi' (u^2+ u_x^2)+\gamma \Phi' y  \Bigr]+ C  e^{-R/6} e^{-\frac{\beta}{6}(t_0-t)} \; .
\label{nini}
\end{equation}
Integrating between $ t$ and $ t_0$  we obtain \eqref{mono} for  any $  t \in [ 0,t_0] $ and $ u $ replaced by $ u_n $ with $ n\ge n_0$. Note that the constant appearing in front of the exponential now also depends on $ \beta$. 
 The convergence results \eqref{cont1}-\eqref{cont2} then ensure that \eqref{mono} holds also for  $ u $ and   $t\in[ 0,t_0] $. Finally, \eqref{mono2} can be proven in exactly the same  way by noticing that  for  $|x-x(t)|<R_0 $ it holds 
  \begin{equation} \label{to2}
   x-z_{t_0}^{-R}(t)=x-x(t)+R+(x(t)-z(t))- (x(t_0)-z(t_0))\ge  -R_0+R+\beta c_0 (t-t_0) \; .
  \end{equation}
\end{proof}
\begin{remark}
It is worth noticing that the definitions of $ \Psi$, $\Phi $, \eqref{po}  and \eqref{defI} ensure that 
\begin{equation}
I^{x_0}_{t_0}(t) \ge \frac{1}{9\pi} \dist{u^2(t)+u_x^2(t)+y(t)}{ \Phi(\cdot-z^{x_0}_{t_0}(t)}, \forall t\in\R \, .
\end{equation}
\end{remark}
\section{Properties of the asymptotic object}
The aim of this section is to prove that for $u _0 \in Y $ satisfying Hypothesis \ref{hyp} and $H^1$-close enough to a peakon, one can extract from the orbit of $u_0 $ an asymptotic object that has a non negative density momentum and gives rise to a $ Y$-almost localized solution.

Let $ u_0 \in Y $ satisfying Hypothesis \ref{hyp}, such that 
\begin{equation}\label{stab}
 \| u_0-c\varphi \|_{H^1} < \Bigl(\frac{\varepsilon^2}{3c^2}\Bigr)^4 \; , \quad 0<\varepsilon<c,
 \end{equation}
 then, according to \cite{CS1} and \cite{E}, 
\begin{equation}\label{stabo}
 \sup_{t\in\R_+} \|u(t)-c\varphi(\cdot-\xi(t))\|_{H^1} <\frac{\varepsilon^2}{c}\; ,
  \end{equation}
 where $ u \in C(\R_+;H^1) $ is the solution emanating from $ u_0$ and  $ \xi(t)\in\R $ is any point where the function $ u(t,\cdot) $ attains its maximum. According to \cite{L}, by the implicit function theorem, we have the following lemma.
  \begin{lemma}\label{modulation}
 There exists $0< \varepsilon_0<1$,  $\kappa_0>0$, $n_0\in \N $  and $ K>1 $ such that if a solution $ u \in C(\R;Y) $ to 
 \eqref{CH} satisfies 
 \begin{equation}\label{gff}
\sup_{t\in\R_+}  \|u(t)-c \varphi(\cdot-z(t)) \|_{H^1} < c \varepsilon_0 \; ,
\end{equation}
for some function $ z \; :\; \R_+\to \R $, then there exists a unique function $ x \; : \R_+\to \R $ such that 
\begin{equation}
\sup_{t\in\R_+} |x(t)-z(t)| < \kappa_0 \;  \label{distxz}
\end{equation}
 and 
\begin{equation}
\int_{\R} u(t) (\rho_{n_0}\ast\varphi')(\cdot-x(t))=0, \quad \forall t\in\R_+ \; , \label{ort}
\end{equation}
where $ \{\rho_n\} $ is defined in \eqref{rho} and where 
$ n_0 $ satisfies : 
\begin{equation}\label{unic}
\forall y\in [-1/2,1/2], \quad \int_{\R} \varphi (\cdot-y)  (\rho_{n_0}\ast\varphi')=0 \Leftrightarrow y=0 \; .
\end{equation}
Moreover, $ x(\cdot)\in C^1(\R) $  with 
\begin{equation}\label{estc}
\sup_{t\in\R_+} |\dot{x}(t)- c| \le \frac{c}{8}
\end{equation}
and if 
  \begin{equation}\label{gf}
\sup_{t\in\R_+}   \|u(t)-c\varphi(\cdot-z(t)) \|_{H^1} <\frac{\varepsilon^2}{c}=c \Bigl(\frac{\varepsilon}{c} \Bigr)^2
\end{equation}
 for $ 0<\varepsilon <c \varepsilon_0 $ then 
\begin{equation}\label{fg}
\sup_{t\in\R_+} \|u(t)-c\varphi(\cdot-x(t))\|_{H^1} \le K \varepsilon \; .
\end{equation}
 \end{lemma}
 At this stage, we fix $ 0<\theta<c $ and we take 
 \begin{equation}\label{defep}
 \varepsilon= \frac{1}{2^4 K}\min \Bigl(\frac{\theta}{2^5}, c \, \varepsilon_0\Bigr)
 \end{equation}
 For $ u_0\in Y $  satisfying Hypothesis  \ref{hyp} and  \eqref{stab} with this $ \varepsilon$, \eqref{stabo}  ensures that \eqref{gff} and  thus \eqref{estc} hold. Moreover, \eqref{fg} is satisfies with 
 $
 K\varepsilon \le \min \Bigl(\frac{\theta}{2^9},\frac{c \varepsilon_0}{2^4}\Bigr) \; 
 $
 so that 
  \begin{equation}\label{Kep}
 \sup_{t\in\R_+} \|u(t)-c\varphi(\cdot-x(t))\|_{H^1} \le \frac{c \varepsilon_0}{2^4}\le \frac{c}{2^4}\; .
 \end{equation}
  It follows that  
  \begin{equation} \label{cc}
  \dot{x}(t) \ge \frac{7}{8} c \;, \quad \forall t\ge 0 .
  \end{equation} 
  and that $u $ satisfies the hypotheses of Lemma \ref{almostdecay} for any $ 0<\alpha<1$ such that 
  \begin{equation}\label{okok}
  (1-\alpha)\ge \frac{\theta}{4c} 
  \end{equation}
  and any  $ 0\le \gamma\le (1-\alpha) c $. In particular, $ u$ satisfies the hypotheses of Lemma \ref{almostdecay} for 
   $ \alpha=1/3$. Note that the hypothesis \eqref{difini} with 
 $$
 \eta_0=\frac{1}{K^8}\min\Bigl( \frac{1}{ 2^{10}},  \frac{\varepsilon_0}{2^4} \Bigr)^8
 $$
 implies that \eqref{stab} holds with $ \varepsilon $ given by \eqref{defep}.
\noindent
\begin{proposition}\label{propasym}
Let $ u_0 \in Y$ satisfying  Hypotheses \ref{hyp} and \eqref{stab} with $\varepsilon$ defined as in \eqref{defep}  and let $u \in C(\R;H^1(\R)) $ the emanating solution of \eqref{CH}. For any sequence $ t_n\nearrow +\infty $ there exists a subsequence $ \{t_{n_k}\}\subset \{t_n\} $ and $ \tilde{u}_0\in Y_+ $ such that 
\begin{equation}\label{pp2}
 u(t_{n_k},\cdot+x(t_{n_k})) \tendsto{n_k\to +\infty} \tilde{u}_0 \mbox { in } H^1_{loc}(\R) 
 \end{equation}
 where $ x(\cdot) $ is a $ C^1$-function satisfying \eqref{ort}, \eqref{estc} and  \eqref{fg}. 
Moreover, the solution of \eqref{CH} emanating from $ \tilde{u}_0 $ is $Y $-almost localized.
\end{proposition}
\begin{proof}
In the sequel we set $ x_0(t)=q(t,x_0) $, $ t\ge 0 $, so that 
\begin{equation}\label{yyyy}
\supp y^{-}(t)\subset ]-\infty,x_0(t)] \quad \text{ and } \quad \supp y^+(t) \subset [x_0(t),+\infty[ , \quad \forall t\ge 0 \; .
\end{equation}
We separate two possible behaviors of $ x_0(\cdot)  $ which has to be treated in  different ways. 
In the first case we can prove that $ x(t)-x_0(t) \to +\infty $ so that for $ t >0 $ large enough we are very close to the case of a non negative density momentum.
 On the other hand, in the second case  $ x(t)-x_0(t) $ is uniformly bounded and we have to use other arguments. In this case we will first prove that the total variation of $ y $ stays bounded for all positive times which enables us to pass to the limit. The non negativity of the limit follows from a decay estimate on the total variation of $ y(t) $ on a growing interval at the left of $ x_0(t)$.\vspace*{4mm} \\
 \noindent
 {\bf Case 1.}  There exists $ t_*\ge 0 $ such that 
\begin{equation}\label{case1}
x_0(t_*) < x(t_*)-\ln (3/2) \; .
\end{equation}
In this case we notice that in view of \eqref{Kep} and \eqref{sobo}, for any $ t\ge 0 $, 
\begin{equation}\label{oo} 
u(t,x) <c \varphi(x-x(t))+ \frac{c}{16} <\frac{2c}{3}+ \frac{c}{16}\le \frac{3c}{4} \quad \text{ for } x\le x(t)-\ln (3/2)
\end{equation}
Therefore, by a continuity argument,  \eqref{defq}, \eqref{cc}, \eqref{oo} and \eqref{case1}  lead to 
\begin{equation}\label{difx0}
x(t)-x_0(t) \ge \frac{c}{8}(t-t_*) +x(t_*)-x_0(t_*)> \frac{c}{8}(t-t_*) +\ln 2  , \quad \forall t \ge t_* \; .
\end{equation}
Applying the almost monotonicity result \eqref{mono2} for $ t\ge t_* $ with $z(t)=x(t)-x_0(t) $  and $ R=x(t_*)-x_0(t_*)>0$ we deduce that there exists $ A_0=A_0(E(u_0), \|y(t_*)\|_{\mathcal M})>0 $ such that 
\begin{equation} \label{A0}
 \Bigl\langle  y(\cdot+x(t), \Phi(\cdot +\frac{c}{8}(t-t_*))\Bigr\rangle   \le A_0 , \quad \forall t\ge t_* \; .
 \end{equation}
 We set 
 $$
  \breve{u}(t,x)=u(t,x) \Phi\Bigl(\cdot-x(t)+\frac{c}{8}(t-t_*)\Bigr) \; ,
  $$
  where $ \Phi $ is defined in \eqref{defPhi}.
  According to the conservation of $ E(\cdot)$, \eqref{difx0} and \eqref{A0}, $ \breve{u} $ is uniformly bounded in $ Y $ for positive time. Therefore for any sequence $ t_n\nearrow +\infty $, there exists $ \tilde{u}_0 \in Y $ and a subsequence of $ \{t_{n_k}\} $  (that we still denote by $t_{n_k} $ to simplify the notation) such that 
\begin{eqnarray*}
\breve{u}(t_{n_k},\cdot +x(t_{n_k})) & \rightharpoonup & \tilde{u}_0 \mbox{ in } H^1(\R)  \\  
\breve{u}(t_{n_k},\cdot +x(t_{n_k})) & \to & \tilde{u}_0 \mbox{ in } H^1_{loc}(\R)  \\  
(\breve{u}-\breve{u}_{xx})(t_{n_k},\cdot +x(t_{n_k})) & \rightharpoonup\hspace{-1mm} \ast & \tilde{y}_0= \tilde{u}_0 - \tilde{u}_{0,xx}  \mbox{ in } {\mathcal M}(\R) \; .\nonumber
\end{eqnarray*}
Since  for $ t\ge t_* $, according to \eqref{difx0} and the support property of $ \Phi $ and $ \Phi'$, $ (\breve{u}-\breve{u}_{xx})(t,\cdot +x(t)) =y(t,\cdot)$ on   $ ]-\frac{c}{8}(t-t_*)+2,+\infty[ $ and thus is a non negative measure on $ ]-\frac{c}{8}(t-t_*)+2,+\infty[ $,  we infer that $ \tilde{u}_0\in Y_+ $. Moreover, the support properties of $ \Phi $ and $\Phi' $ imply that 
\begin{equation}
u(t_{n_k},\cdot +x(t_{n_k}))  \to  \tilde{u}_0 \mbox{ in } H^1_{loc}(\R)  \label{vc} 
\end{equation}
and, for all $ A>0 $ and $ \phi\in C_0(\R) $ with $ \supp \phi \subset [-A,+\infty[ $, 
\begin{equation}\label{vcc}
 \Bigl\langle y(t_{n_k},\cdot +x(t_{n_k})), \phi \Bigr\rangle \to   \Bigr\langle\tilde{y}_0, \phi\Bigr\rangle \; .
 \end{equation}
  Since, by  \eqref{estc}, $ \{x(t_n+\cdot)-x(t_n)\} $ is uniformly equi-continuous, Arzela-Ascoli theorem ensures that there exists a subsequence
   $ \{t_{n_k}\} \subset \{t_n\} $ and $ \tilde{x}\in C(\R) $ such that for all $ T>0 $,
   \begin{equation}\label{cvx}
   x(t_{n_k}+\cdot)-x(t_{n_k}) \tendsto{t\to+\infty} \tilde{x} \mbox{ in } C([0,T]) \; .
   \end{equation}
  Therefore, 
 on account of  \eqref{vc}, \eqref{cvx} and part {\bf 3.} of Proposition \ref{prop1} for any $ t\in \R $,
\begin{eqnarray}
u(t_{n_k}+t,\cdot +x(t_{n_k}+t)) & \to & \tilde{u}(t,\cdot+\tilde{x}(t)) \mbox{ in } H^1_{loc}(\R)   \label{strongcv}
\end{eqnarray}
where $ \tilde{u} \in C(\R_+;H^1(\R)) $ is the solution of \eqref{CH} emanating from $ \tilde{u}_0\in Y_+ $. 
Moreover, for all $ A>0 $ and  $ \phi \in C_0(\R) $ with $ \supp \phi \subset [-A,+\infty[ $,  it holds 
\begin{equation}
\dist{y(t_{n_k}+t,\cdot +x(t_{n_k}+t))}{\phi} \rightarrow  \dist{\tilde{y}(t,\cdot+\tilde{x}(t))}{\phi}  \label{weakcvy}\; ,
\end{equation}
where $ \tilde{y}=\tilde{u}-\tilde{u}_{xx} $. Indeed, on one hand, it follows from  part {\bf 3.} of Proposition \ref{prop1}  that 
$$
\dist{y(t_{n_k}+t,\cdot +x(t_{n_k})+\tilde{x}(t))}{\phi} \rightarrow  \dist{\tilde{y}(t,\cdot+\tilde{x}(t))}{\phi} 
$$
and on the other hand, the uniform continuity of $ \phi $ together with \eqref{cvx} ensure that 
\begin{align*}
&\dist{y(t_{n_k}+t,\cdot +x(t_{n_k})+\tilde{x}(t))-y(t_{n_k}+t,\cdot +x(t_{n_k}+t))}{ \phi}  \\
&= \dist{y(t_{n_k}+t)}{\phi(\cdot-x(t_{n_k})-\tilde{x}(t))-\phi(\cdot-x(t_{n_k}+t))}\to 0 
\end{align*}
In view of \eqref{strongcv}  we infer that $ (\tilde{u}, \tilde{x}(\cdot)) $  satisfies   \eqref{ort} and \eqref{fg} 
with the same $ \varepsilon $ than $ (u,x(\cdot))$.   Therefore,  \eqref{defep} 
 forces    $ (\tilde{u}, \tilde{x}(\cdot)) $ to satisfy \eqref{gff} and the  uniqueness result in Lemma \ref{modulation} ensures that $ \tilde{x}(\cdot) $ is a $ C^1$-function  and  satisfies \eqref{estc}.

Let us now prove that $ \tilde{u} $ is $Y$-almost localized. In the sequel, for $ u\in Y $ and $ \gamma\ge 0 $, we defined the quantity $ G(u) $ by 
  $$
   G_{\gamma}(u)=E(u)+\frac{c}{4} \dist{u-u_{xx}}{1} \; .
   $$
   We  will also 
make use of the following functionals that measure the quantity $ G(u)$ at the right and at the left of $ u$.
For $ 0\le \gamma\le c 2^{-4} $, $ v\in Y$ and $ R>0 $ we set 
 \begin{equation}\label{defJr}
 J_{\gamma,r}^{R}(v)=\dist{v^2+v_x^2}{\Psi(\cdot -R)} +\gamma \dist{v-v_{xx}}{\Phi(\cdot-R)}
 \end{equation}
 and
 \begin{equation}\label{defJl}
 J_{\gamma, l}^R(v)=\dist{v^2+v_x^2}{(1-\Psi(\cdot+R))} +\gamma  \dist{v-v_{xx}}{1-\Phi(\cdot+R)}\; .
 \end{equation}
 We  separate  $ G(v) $ into  two parts : 
\begin{align*}
G_{o,\gamma}^{R}(v) &=\dist{v^2+v_x^2}{1-\Psi(\cdot+R)+\Psi(\cdot-R)} +\gamma \dist{v-v_{xx})}{1-\Phi(\cdot+R)+\Phi(\cdot-R)}\\
 & =J_{\gamma,r}^R(v)+J_{\gamma,l}^R(v) \; ,
\end{align*}
which  almost ``localizes''  outside  the ball of radius $ R $ and 
\begin{align*}
G_{i,\gamma}^{R}(v) & =\dist{v^2+v_x^2}{\Psi(\cdot+R)-\Psi(\cdot-R)}+\gamma \dist{v-v_{xx}}{\Phi(\cdot+R)-\Phi(\cdot-R)} \\
 & =G(v)-G_o^R(v) \; ,
 \end{align*}
which almost ``localizes'' inside this ball. It is however worth mentioning that, in view of \eqref{defPhi},
 the function $\Phi(\cdot+R)-\Phi(\cdot-R) $ is  indeed supported in $ [-R,R+1]$. 
We first notice that the almost monotonicity result \eqref{mono} ensures that for any $ \varepsilon>0 $ there exists $ R_\varepsilon'>0 $ such that 
 \begin{equation}\label{decR}
 J_{\gamma,r}^{R_\varepsilon'}(\tilde{u}(t, \cdot +\tilde{x}(t))) < \varepsilon  , \quad \forall t\ge 0 \, .
 \end{equation}
 Indeed, let $ t_0>0 $ be fixed.  Fixing $\alpha=\beta=1/4 $ and taking $ z(\cdot)=(1-\alpha) x(\cdot)$, $z(\cdot) $ clearly satisfies \eqref{condz}
  and \eqref{fg}-\eqref{difx0} ensure that it also satisfies \eqref{zo1} for $ R\ge 2$. Moreover, we have  $ J_{\gamma,r}^{R}(u(t_0, \cdot +x(t_0))=I^{+R}_{t_0}(t_0) $
  where $I^{+R}_{t_0} $ is defined in \eqref{defI}.  Since obviously, 
 $$
 J_{\gamma,r}^R \Bigl(u(t,\cdot+x(t))\Bigr)\ge I^{+R}_{t_0}(t)\; , \quad \forall 0\le t\le t_0,
 $$
 we deduce from \eqref{mono} that 
 \begin{equation}\label{monoJr}
 J_{\gamma,r}^R \Bigl(u(t_0,\cdot+x(t_0))\Bigr)\le J_{\gamma,r}^R\Bigl(u(t,\cdot+x(t))\Bigr)+K_0 e^{-R/6} \;  , \quad \forall 0\le t\le t_0,
 \end{equation}
 where $ K_0 $ is the constant appearing in \eqref{mono}.  Therefore, taking $ R_\varepsilon' \ge 2 $ such that $ K_0 e^{-R_\varepsilon/6} <\varepsilon/2 $ and 
 $ J_r^{R_\varepsilon'}(u(0, \cdot+x(0)) <\varepsilon/2 $ we get that for all $ t \ge 0 $
 \begin{equation}\label{dede}
  J_{\gamma,r}^{R_\varepsilon'} \Bigl(u(t,\cdot+x(t))\Bigr)\le   J_{\gamma,r}^{R_\varepsilon'} \Bigl(u(0,\cdot+x(0))\Bigr) +\varepsilon/2 \le \varepsilon 
 \end{equation}
 Passing to the weak limit, this leads to \eqref{decR}.  Moreover, \eqref{difx0} ensures that for any $ R >0 $ there exists $ t(R)>0 $ such that 
 $$ x_0(t) \le  x(t_R)-R+\frac{9}{10}(x(t)-x(t_R) , \quad \forall t\ge t(R)\; .
 $$
 Therefore, the hypotheses \eqref{condz} and \eqref{zo2} of Lemma \ref{almostdecay} are fulfilled for $z(t)=x(t_R)-R+\frac{9}{10}(x(t)-x(t_R))$ and $ t\ge t_R $ and, proceeding as above, \eqref{mono2} leads to 
  \begin{equation}\label{monoJl}
 J_{\gamma,l}^R \Bigl(u(t,\cdot+x(t))\Bigr)\ge J_{\gamma,l}^R\Bigl(u(t(R),\cdot+x(t(R)))\Bigr)-K_0 e^{-R/6} \;  , \quad \forall  t\ge t(R)
 \end{equation}
 
Now we notice that to prove $ Y$-almost localization of $ \tilde{u} $, it suffices to prove that  for all $\varepsilon>0 $, there exists $ R_\varepsilon>0 $ such that
\begin{equation}\label{po2}
G_{o,\gamma}^{R_\varepsilon}\Bigl(\tilde{u}(t,\cdot+\tilde{x}(t))\Bigr)<\varepsilon \; , \quad \forall t\in\R \, .
\end{equation}
 \noindent
 Indeed if \eqref{po2} is true for some $ (\varepsilon, R_\varepsilon) $ then $ (\tilde{u},\tilde{x}) $ satisfies \eqref{defloc} with $(\varepsilon/2, 2 R_\varepsilon) $.
 As indicated above, we prove \eqref{po2} by contradiction. Assuming that \eqref{po2} is not true,  there exists $ \varepsilon_0>0 $ such that for any $ R>0 $ there exists $ t_R\in \R $ satisfying 
  \begin{equation}
  G_{o,\gamma}^{R}\Bigl( \tilde{u}(t_R, \cdot+\tilde{x}(t_R))\Bigr)\ge \varepsilon_0
  \end{equation}
  Let $ R_0>R_{\frac{\varepsilon}{10}}' $ such that 
   \begin{equation}
  G_{o,\gamma}^{R_0}\Bigl( \tilde{u}(0)\Bigr)\le \frac{\varepsilon_0}{10} \quad \text{ and } \quad K_0 e^{-R_0/6}<\frac{\varepsilon_0}{10} \; .
  \end{equation}
   The conservation of $G $ then forces
  $$
   G_{i,\gamma}^{R_0} (\tilde{u}(t_{R_0}, \cdot+\tilde{x}(t_{R_0}))\le G^{R_0}_{i,\gamma}(\tilde{u}(0))-\frac{9}{10} \varepsilon_0 \; .
  $$
  Recalling that $ \Psi(\cdot+R)-\Psi(\cdot-R) $ and $ \Phi(\cdot+R)-\Phi(\cdot-R) $ belong to $ C_0(\R) $,  the convergence results \eqref{strongcv}-\eqref{weakcvy} ensure that for $ k\ge k_0$ with $ k_0 $ large enough,
  $$
  G_{i,\gamma}^{R_0} (u(t_{n_k}+t_{R_0}), \cdot+x(t_{n_k}+t_{R_0}))\le G^{R_0}_{i,\gamma}(u(t_{n_k}, \cdot+x(t_{n_k})))-\frac{4}{5} \varepsilon_0 \; .
  $$
  We first assume that $ t_{R_0}>0 $. By \eqref{decR} and the conservation of $G $ this ensures that 
  \begin{equation}\label{bb}
  J_{\gamma,l}^{R_0}(u(t_{n_k}+t_{R_0}), \cdot +x(t_{n_k}+t_{R_0}) \ge  J_{\gamma,l}^{R_0}(u(t_{n_k}, \cdot +x(t_{n_k})) +\frac{7}{10} \varepsilon_0 \; .
  \end{equation}
  Now we take a subsequence $\{t_{n_k'}\}$ of $ \{t_{n_k}\} $ such that $ t_{n_0'} \ge t(R_0)$, $ t_{n_{k+1}'}-t_{n_k'} \ge t_{R_0}$ and $ n_k'\ge n_{k_0} $. From \eqref{bb} and again \eqref{monoJl},  we get that for any $ k\ge 0 $, 
  $$
  J_{\gamma,l}^{R_0}(u(t_{n_k'}, \cdot +x(t_{n_k'}) )\ge  J_{\gamma,l}^{R_0}(u(t_{n_0'}, \cdot +x(t_{n_0'})) +\frac{3}{5} k\,  \varepsilon_0 \tendsto{k\to +\infty} +\infty \; 
  $$
  that contradicts the conservation of $ G $  in view of \eqref{dede} and   $ G^{R_0}_{i,\gamma}(u(t_{n_k'}, \cdot+x(t_{n_k'})))\to G^{R_0}_{i,\gamma}(\tilde{u}(0, \cdot+\tilde{x}_0))) $. 
  Finally, if $ t_{R_0} <0 $, then for $ k\ge k_0 $ such that  $ t_{n_k}>|t_{R_0}| $  we get in the same way 
  $$
   0\le J_{\gamma,r}^{R_0}\Bigl(u(t_{n_k}, \cdot +x(t_{n_k}) \Bigr)\le  J_{\gamma,r}^{R_0}\Bigl(u(t_{n_k}-|t_{R_0}|, \cdot +x(t_{n_k}-|t_{R_0}|))\Bigr) -\frac{7}{10} \varepsilon_0 \; .
   $$
  that  contradicts \eqref{decR} and $ R_0>R_{\frac{\varepsilon}{10}}' $ for $ k \ge k_0 $ large enough. This proves the $ Y$-almost localization of $ \tilde{u} $.

{\bf Case 2.} For all $ t\ge 0 $ it holds 
\begin{equation}\label{case2}
x_0(t)\ge x(t)-\ln (3/2) \; .
\end{equation}
Noticing that  \eqref{dodo} leads to $ u_x(t,\cdot)\ge u(,\cdot)$ on $ ]-\infty, x_0(t)[ $ and that \eqref{Kep}, \eqref{sobo} and \eqref{case2} ensure that 
 $  u(t,x_0(t)-\ln 2)\ge c \varphi(-\ln 3) -\frac{c}{16}\ge \frac{13c}{48}  $, we infer that 
\begin{equation}\label{cc22}
 u(t) \ge  \frac{13}{48} c \text{ on } [x_0(t)-\ln 2, x_0(t)] , \quad \forall t\ge 0 \; .
\end{equation}

We claim that
\begin{equation}
\|y(t)\|_{\mathcal M} \le 4 \Bigl(2+ \frac{9}{c}  \sqrt{E(u_0)}  \Bigr) \|y_0\|_{\mathcal M} , \quad \forall t\ge 0  \; . \label{claimy}
\end{equation}
and 
\begin{equation}
\|y^-(t)\|_{\mathcal M(]x(t)-\frac{c}{96}t,+\infty[)}\tendsto{t\to +\infty} 0  \; . \label{claimy2}
\end{equation}
To prove this claim we approximate $ u_0 $ by $ \{u_{0,n}\} $ as in \eqref{app} and work with the global solutions $ u_n $ emanating from $ u_{0,n} $ that satisfy Hypothesis \ref{hyp} with the same $ x_0 $. We define $ q_n $ as the flow associated with $ u_n $ and we set 
$$
 x_{0,n}(t)=q_n(t,x_0) 
 $$
 Let us recall that, in \cite{C}, it is shown that  for any  $ (t,x)\in \R_+\times \R $, 
  \begin{equation}\label{yy}
  y_n(0,x)=y_n(t,q_n(t,x)) q_{n,x}(t,x)^2 
  \end{equation}
  with
\begin{equation}\label{formula1}
q_{n,x}(t,x) =\exp\Bigl( \int_0^t u_x(s,q(s,x))\,ds\Bigr) \; .
\end{equation}
 Let $ t_0>0 $ be fixed. According to \eqref{cont1}  and \eqref{cc22} there exists $ n_0\ge 0 $ such that for all $ n\ge n_0 $, 
\begin{equation}\label{fc2}
 u_n(t) \ge \frac{c}{4}  \text{ on } [x_{0,n}(t)-\ln 2, x_{0,n}(t)] , \quad \forall t\in [0,t_0]
  \end{equation}
  We proceed in three steps.\\
  {\it Step 1.} 
In this step, we prove that   for all $t\in [0,t_0]$ and $n $ large enough, 
\begin{equation}
\Bigl|  \displaystyle \int_{x_{0,n}(t)-\ln 2}^{x_{0,n}(t)} y_n(t,s) \, ds  \Bigr| \le \exp(- \frac{c}{4} t ) \|  y_{0,n}\|_{L^1}  \; . \label{claimx0}
\end{equation}
Let $ t\in [0,t_0] $ and $ x\in [x_{0,n}(t)-\ln 2, x_{0,n}(t)]  $, then one has, $ \forall 0 \le \tau \le t $, 
$$
q_n(\tau,q_n^{-1}(t,x))\le q_n(\tau,q_n^{-1}(t,x_{0,n}(t))=x_{0,n}(\tau)
$$
where  for all $ t\ge 0 $, $ q_n^{-1}(t,\cdot) $ is the inverse mapping of $ q_n(t,\cdot) $.
Since, according to \eqref{dodo} and \eqref{fc2},  $u_{n,x}\ge u_n \ge \frac{c}{4} $ on $ [x_{0,n}(\tau)-\ln 2, x_{0,n}(\tau)]  $ it holds 
$$
u_n(\tau,x_{0,n}(\tau))\ge u_n\Bigl(\tau,q_n(\tau,q_n^{-1}(t,x))\Bigr) \ge 0 
$$
and thus 
$$
\frac{d}{d\tau} \Bigl( x_{0,n}(\tau)-q_n(\tau, q_n^{-1}(t,x)\Bigr) \ge 0 \quad \text{ on } [0,t]\; .
$$
Therefore, for all $ \tau\in [0,t] $ one has 
$$
q_n(\tau,q_n^{-1}(t,x))\in [x_{0,n}(\tau)-\ln 2, x_{0,n}(\tau)]
$$
which ensures that 
$$
u_{n,x}\Bigl(\tau,q_n(\tau,q_n^{-1}(t,x))\Bigr)\ge \frac{c}{4} 
$$
and \eqref{formula1} leads to 
$$
q_{n,x}(t,q_n^{-1}(t,x))=\exp\Bigl( \int_0^t u_{n,x}\Bigl(\tau,q_n(\tau,q_n^{-1}(t,x))\Bigr)\, d\tau\Bigr) \ge \exp( \frac{c}{4} t) \; .
$$
Therefore for all $ t\in [0,t_0] $, \eqref{yy} leads to
\begin{align*}
-\int_{x_{0,n}(t)-\ln 2}^{x_{0,n}(t)} y_n(t,x) \, dx &= -\displaystyle \int_{q_n^{-1}(t,x_{0,n}(t)-\ln 2)}^{x_{0}} y_n(t,q_n(t,\theta)) q_{n,x}(t,\theta)\, d\theta \\
 & \le -e^{-\frac{c}{4}t}\displaystyle  \int_{q_n^{-1}(t,x_{0}(t)-\ln 2)}^{x_{0}} 
y_n(t,q_n(t,\theta)) q_{n,x}^2(t,\theta)\, d\theta \\
& \le -e^{-\frac{c}{4}t}\displaystyle \int_{x_{0}-\ln 2}^{x_{0}}  y_n(0,\theta) \, d\theta
\end{align*}
which proves  \eqref{claimx0}.\\
{\it Step 2.} In this step we prove that for $ t\in[0,t_0] $ and $ n\ge 0  $ large enough,
\begin{equation}\label{step2}
 \|y_n(t) \|_{L^1} \le 2 \Bigl(2+ \frac{9}{c}  \sqrt{E(u_{0,n})}  \Bigr) \|y_{0,n}\|_{\mathcal M} \; .
 \end{equation}
Let $ \Phi\; : \R\to \R_+  $ defined by  $ \Phi\equiv 0 $ on $ ]-\infty,-\ln 2] $,  $ \Phi=1 $ on $ \R_+$ and $ 
\Phi(x)= (x+\ln 2)/\ln 2 $ for $x\in [-\ln 2,0] $.   Then 
\begin{eqnarray}
\frac{d}{dt} \int_{\R} y_n(t) \Phi(\cdot-x_{0,n}(t)) & =&  -\dot{x}_{0,n}(t) \int_{\R} y \Phi'(\cdot-x_{0,n}(t)) +  \int_{\R} u_n  y_n\Phi'(\cdot-x_{0,n}(t)) \nonumber \\
& & + \int_{\R} (u_n^2-u_{n,x}^2) \Phi'(\cdot-x_{0,n}(t)) \label{gv}
\end{eqnarray}
Since $ \supp \Phi'(\cdot-x_{0,n}(t)) \subset [x_{0,n}(t)-\ln 2, x_{0,n}(t)] $,  \eqref{dodo} ensures that the last term of the right-hand side of \eqref{gv} is non positive 
 and \eqref{claimx0} ensures that 
 \begin{eqnarray}
\Bigl|   \int_{\R} (u_n(t)-\dot{x}_{0,n}(t))  y_n(t) \Phi'(\cdot-x_{0,n}(t))\Bigr| &\le &  \sqrt{2 E(u_{0,n})} \,\Bigl|   \int_{\R}   y_n(t) \Phi'(\cdot-x_{0,n}(t))\Bigr|\nonumber\\
& \le & \frac{\sqrt{2}}{\ln 2} \sqrt{E(u_{0,n})} \|y_{0,n}\|_{L^1}  \exp(-\frac{c}{4}t) \; .
 \end{eqnarray}
 Therefore \eqref{gv} leads to 
 \begin{eqnarray}
 \int_{\R} y_n(t) \Phi(\cdot-x_{0,n}(t)) &\le & \int_{\R} y_{0,n} \Phi(\cdot-x_{0,n}) + \frac{8}{c}  \sqrt{E(u_{0,n})} \|y_{0,n}\|_{L^1}  \nonumber \\
 & \le & \Bigl(1+ \frac{4\sqrt{2}}{c\, \ln 2}  \sqrt{E(u_{0,n})}  \Bigr) \|y_{0,n}\|_{L^1} 
\end{eqnarray}
Gathering this last estimate with \eqref{claimx0} we get that 
$$
 \int_{x_{0,n}(t)}^{+\infty}  y_n(t,s)\, ds 
  \le  \Bigl(2+  \frac{4\sqrt{2}}{c\, \ln 2}  \sqrt{E(u_{0,n})}  \Bigr) \|y_{0,n}\|_{\mathcal M} 
$$
 that leads to \eqref{step2} by making use of the conservation of  $ M(\cdot) $. 
   \eqref{claimy} then follows  by  passing to the limit in $ n$ and  using \eqref{cont2}.\\
  {\it Step 3.} Finally we prove that  for all $t\in [0,t_0]$ and $ n\ge 0 $ large enough,
  \begin{equation}\label{step3}
  \|y_{n}^-(t)\|_{\M]x_0(t)-\frac{c}{8} t, +\infty[} \le  \exp(-\frac{c}{4}t)\|y_{0,n} \|_{L^1} \; .
  \end{equation}
  For this we first notice that since  \eqref{dodo} leads to $ u_{x}(t,\cdot)\ge u(t,\cdot)$ on $ ]-\infty, x_{0}(t)[ $, it holds  
$$
u(t,x) \le e^{x-x_{0}(t)} u(t,x_0(t)) , \quad \quad \forall x\le x_{0}(t) \; .
$$
Therefore,  \eqref{Kep} and  \eqref{sobo} force 
$$
u(t, x) \le \frac{1}{2} u(t,x_0(t)) \le \frac{1}{2} \frac{17}{16} c \le \frac{17}{32} c  , \quad \forall (t,x]\in \R_+\times ]-\infty, x_0(t)-\ln 2]\; .
$$
On the other hand, $ u_x(t)\ge u(t) $ on $ ]-\infty,x_0(t)] $ and \eqref{case2}  forces 
$$
u(t,x_0(t)) \ge \frac{2c}{3} -\frac{c}{16}= \frac{29c}{48}\; , \quad \forall t\ge 0 .
$$
Moreover, \eqref{Kep} ensures that $ u\ge - 2^{-4} c $ on $\R_+\times \R $. 
Taking $ n $ large enough we can thus assume that 
\begin{equation}\label{pl1}
u_n(t, x) \le \frac{9 c}{16}  , \quad \forall (t,x)\in [0,t_0]\times]-\infty, x_{0,n}(t)-\ln 2] ,
\end{equation}
\begin{equation}\label{pl2}
u_n(t, x_{0,n}(t)) \ge \frac{7 c}{12} , \quad \forall t\in [0,t_0] \; . 
\end{equation}
and 
\begin{equation}\label{pl3}
u_n(t, x) \ge -\frac{c}{8} \quad \text{on} \quad [0,t_0]\times \R \; . 
\end{equation}
For any $ t_1\ge 0 $ we define the function $ q_{n,t_1} $ on $ \R_+\times\R $ by 
 \begin{equation}\label{defqn}
  \left\{ 
  \begin{array}{rcll}
  \partial_t q_{n,t_1} (t,x) & = & u_n(t,q_{n,t_1}(t,x))\, &, \; (t,x)\in \R_+\times \R \\
  q_{n,t_1}(t_1,x) & =& x\, & , \; x\in\R \; 
  \end{array}
  \right. .
  \end{equation} 
  $q_{n,t_1} $ is the flow associated with $ u_n $ that satisfies $ q_{n,t_1}(t_1)=Id $. 
 Obviously  $ q_n=q_{n,0} $ and one can easily check that for $ q_{n,t_1} $ \eqref{formula1} becomes 
$$
\partial_x q_{n,t_1}(t,x) =\exp\Bigl( \int_{t_1}^t u_{n,x}(s,q_{n,t_1}(s,x))\, ds\Bigr) , \quad \forall t\ge 0, 
$$
so that \eqref{pl3} and $ u_{n,x}\ge u_n $ on $]-\infty,x_{0,n}(t)] $  ensure that 
\begin{equation}\label{dw}
\partial_x q_{n,t/2}(t,x)\ge \exp(-\frac{c}{16}t), \quad \forall (t,x)\in [0,t_0]\times \R \; .
\end{equation}
 For any $ t\ge 0 $, we denote by $ q_{n,t_1}^{-1}(t,\cdot) $ the inverse mapping of $q_{n,t_1}(t)$.
By a continuity argument, \eqref{pl1}-\eqref{pl2} and \eqref{defq} lead to 
$$
x_{0,n}(t)-q_{n,t/2}(t,x_{0,n}(t/2)-\ln 2) \ge \ln 2 + \frac{c}{96} t \quad .
$$
Therefore, for any $x\in [x_{0,n}(t)-\ln 2 - \frac{c}{96} t,x_{0,n}(t)] $, we have $q^{-1}_{n,t/2}(t,x)\in  [x_{0,n}(t/2)-\ln 2,x_{0,n}]$ and thus 
for all $ t\in [0,t_0] $,  a change of variables along the flow, \eqref{yy}, \eqref{dw} and \eqref{claimx0} lead to
\begin{align*}
-\int_{x_{0,n}(t)-\ln 2 -\frac{c}{96} t,x_{0,n}(t)}^{x_{0,n}(t)} y_n(t,x) \, dx &=
  -\displaystyle \int_{q^{-1}_{n,t/2}(t,x_{0,n}(t)-\ln 2 -\frac{c}{96} t)}^{x_{0,n}(t/2)} y_n(t,q_{n,t/2}(t,\theta)) \partial_x q_{n,t/2}(t,\theta)\, d\theta \\
 & \le -e^{\frac{c}{16}t}\displaystyle \int_{x_{0,n}(t/2)-\ln 2}^{x_{0,n}(t/2)} 
y_n(t,t/2,q_{n,t/2}(t,\theta)) [\partial_x q_{n,t/2}(t,\theta)]^2\, d\theta \\
& \le -e^{\frac{c}{16}t}\displaystyle \int_{x_{0,n}(t/2)-\ln 2}^{x_{0,n}(t/2)}  y_n(t/2,\theta) \, d\theta\\
& \le  e^{\frac{c}{16}t} e^{-\frac{c}{8}t}   \|y_{0,n}\|_{L^1} \le  e^{-\frac{c}{16}t}  \|y_{0,n}\|_{L^1}\;.
\end{align*}
This proves that 
$$
\|y_n^-(t) \|_{L^1(]x_{0,n}(t)-\ln 2 -\frac{c}{96} t, +\infty[)} \le e^{-\frac{c}{16}t}  \|y_{0,n}\|_{L^1}\;.
$$
Now we notice that \eqref{cont2} ensures that $ y_n^-(t)   \rightharpoonup \! \ast \; z $ in $\M$ with $ z\in \M^+(\R) $ and $ 0\le y^-(t)\le z $. Therefore, 
passing to the limit in $ n \to \infty $  and then in $ t_0\to +\infty $, we obtain that 
$$
\|y^-(t) \|_{\M(]x_0(t)-\ln(3/2)-\frac{c}{96}t,+\infty[)}\le  e^{-\frac{c}{16}t} \|y_0\|_{\M} , \quad \forall t\ge 0 .
$$
Finally, \eqref{Kep}, \eqref{sobo} and $ u_x\ge u $ on $ ]-\infty,x_0(t)[ $ ensure that 
\begin{equation}\label{fl4}
x_0(t) \le x(t)+\ln (3/2) , \quad \forall t\ge 0 ,
\end{equation}
 which completes the proof of \eqref{claimy2}. 

Now, in view of \eqref{claimy},  for any sequence $ t_n\nearrow +\infty $, there exists $ \tilde{u}_0 \in Y $ and a subsequence of $ \{t_{n}\} $  (that we still denote by $t_{n} $ to simplify the notation) such that 
\begin{eqnarray}
u(t_{n},\cdot +x(t_{n})) & \to & \tilde{u}_0 \mbox{ in } H^1_{loc}(\R) \label{442} \\  
y(t_{n},\cdot +x(t_{n})) & \rightharpoonup\hspace{-1mm} \ast & \tilde{y}_0= \tilde{u}_0 - \tilde{u}_{0,xx}  \mbox{ in } {\mathcal M}(\R) \; .\label{cvy}
\end{eqnarray}
and \eqref{case2} and \eqref{claimy2} ensure that $ \tilde{y_0} $ is a non negative bounded measure. Finally, according to part 3. of Proposition \ref{prop1}  and \eqref{claimy},  the solution ${\tilde u} $emanating from 
 $ \tilde{u}_0 $ belongs to $ C_b(\R;Y_+) $.   
  To prove the $ Y$-almost localization of the solution $ \tilde{u} $ emanating from $ \tilde{u}_0$, we proceed as in the Case 1.  We first obtain as in Case 1 that there exists 
   a $ C^1$-function   $ \tilde{x} $ such that \eqref{cvx}, \eqref{strongcv} hold  for some subsequence $\{t_{n_k}\} $ of $\{t_n\} $. Moreover, according to the  boundedness of $\{\|y_n\|_{\M}\} $ we may also require that for any $ \phi\in C_0(\R) $ and any $ t\in \R $, 
   \begin{equation}
\dist{y(t_{n_k}+t,\cdot +x(t_{n_k}+t))}{\phi} \rightarrow  \dist{\tilde{y}(t,\cdot+\tilde{x}(t))}{\phi}  \label{weakcvy2}\; .
\end{equation}
But then \eqref{case2}, \eqref{claimy2} force
  \begin{equation}\label{sup}
   \supp  \tilde{y}(t) \subset  [\tilde{x}(t)-\ln 2,+\infty[ , \quad \forall t\in \R \; .
\end{equation}
Now, according to  \eqref{fl4},  $ \supp y^-(t)\subset ]-\infty,x(t)+\ln (3/2)] $ for all $ t\ge 0$. Therefore, for $ R>0 $ big enough, we can thus apply Lemma \ref{almostdecay} to obtain that 
   $ J^R_{r,\gamma}(u(t,\cdot +x(t))) $ is almost non increasing which ensures that \eqref{decR} holds.  Then, proceeding exactly as in  the case 1, but with $ \gamma=0$, we obtain that $ \tilde{u} $ is $ H^1$-almost localized, i.e. for any $ \varepsilon>0 $ there exists $ R_\varepsilon>0 $ such that 
   $$
   G_{o,0}^{R_\varepsilon} (\tilde{u}(t,\cdot+\tilde{x}(t))) <\varepsilon , \quad \forall t\in \R \, .
   $$
  This last estimate together with \eqref{decR} and \eqref{sup} prove the $Y$-localization of $ \tilde{u} $.
    \end{proof}
    \section{Asymptotic stability  results}
    \subsection{Asymptotic stability of a peakon}
   With Proposition \ref{propasym} in hands, the proof of the asymptotic stability of a single peakon follows very closely the proof in \cite{L} since the rest of the proof 
   only uses Theorem \ref{liouville} and the monotonicity results 
    for $\gamma=0$ (i.e. only $E(\cdot) $ is involved). 
   
   Let $ u_0 \in Y$ satisfying  hypothesis \ref{hyp} and \eqref{stab} with $\varepsilon$ defined as in \eqref{defep} and let $ \{t_n\} $ be an increasing sequence that tends to 
   $ +\infty $. Combining  Theorem \ref{liouville} and Proposition \ref{propasym} we infer that there exists a subsequence $ \{t_{n_k}\} \subset  \{t_n\} $, $x_0\in\R $ and $ c_0>0  $ close to $ c$ such that 
   \begin{equation}\label{zz}
 u(t_{n_k},\cdot+x(t_{n_k})) \tendsto{n_k\to +\infty} \tilde{u}_0=\varphi_{c_0}(\cdot-x_0) \mbox { in } H^1_{loc}(\R) 
 \end{equation}
 where $ x(\cdot) $ is a $ C^1$-function satisfying \eqref{ort}, \eqref{estc} and \eqref{fg}. Moreover, \eqref{fg}, \eqref{defep} and \eqref{zz} ensure that  $ |x_0|\ll 1/2 $. 
 
 Since by \eqref{zz}, $\tilde{u}_0 $ satisfies the orthogonality condition \eqref{ort},
 \eqref{unic} forces $ x_0=0 $. On the other hand, \eqref{pp2} ensures that $\displaystyle c_0=\lim_{n\to +\infty} \max_{\R} u(t_n') $ and thus 
  $$
  u(t_n',\cdot+x(t_n')) -\lambda(t_n')\varphi  \tendsto{n\to +\infty} 0  \mbox { in } H^1_{loc}(\R) 
  $$
where we set $
 \lambda(t):=\max_{\R} u(t) , \quad \forall t\in\R $. Since this is the only possible limit, it follows that 
  \begin{equation}\label{pp3}
 u(t,\cdot+x(t))-\lambda(t)\varphi  \tendsto{t\to 0} 0 \mbox { in } H^1_{loc}(\R) 
 \end{equation}
The convergence of the scaling parameter $ \lambda(t) $  and of the derivative function $ \dot{x} $ towards $ c_0 $ at $ +\infty $, as well as the strong $ H^1$-convergence on $ ]\theta t,+\infty[ $ follow exactly as in (\cite{L}, Section 5) and will thus be omitted.
 
  It remains to prove the $ H^1$-convergence in $ ]-\infty,-\theta t[ $. Note that such convergence at the left was not established  in \cite{L}.  In the appendix we complete the result in \cite{L} by proving that for  $u_0\in Y_+ $ the energy at the left of any given point decays to zero as time goes to $+\infty$. In the present case the  desired $H^1$-convergence is a direct consequence of the following monotony result : Let $ \Psi $ be the function defined in \eqref{defPsi}, then for 
any $ z\in \R $, the functional 
$$
\Lambda_z \; : \; t\mapsto  \int_{\R} \Psi(\cdot +\frac{\theta}{2} t -z) (u^2(t)+u_x^2(t))
$$
is not increasing on $ \R_+$. Indeed, one has clearly $ \lim_{z\to-\infty} \Lambda_z(0)=0  $ and for any given $z\in\R$, $ -\theta t <-\frac{\theta}{2} t +z $ for $ t\ge 0 $ large enough. Therefore the monotony of $ \Lambda_z $ ensures that $\|u(t)\|_{H^1(x<\theta t)}\to 0 $ as $ t\to +\infty $ and the result follows since \eqref{cc} obviously forces  as well that $\|\varphi(\cdot-x(t))\|_{H^1(x<\theta t)}\to 0 $ as $ t\to +\infty $. 

By continuity with respect to initial data in $ H^1 $, it suffices prove this monotony result for smooth solutions. Then according to \eqref{go}  it holds 
$$
\frac{d}{dt} \Lambda_z(t) = \theta \int_{\R} (u^2+u_x^2) \Psi'(\cdot +\frac{\theta}{2}t -z) +\int_{\R} u u_x^2  \Psi'(\cdot +\frac{\theta}{2}t -z)
+2\int_{\R} uh  \Psi'(\cdot +\frac{\theta}{2}t -z)
$$
where $ h=(1-\partial_x^2)^{-1}(u^2+u_x^2/2)\ge 0  $. 
  
Now, we notice that by one hand , \eqref{difini}, \eqref{stabo},  \eqref{sobo}  and the positivity of $ \varphi $ yield 
  $$
  u(t) \ge -\frac{\theta^2}{4c} \ge - \frac{\theta}{8} ,\quad \forall t\ge 0 \; .
  $$
  On the other hand, we have   $\int_{\R} h  \Psi'=\int_{\R} (u^2+u_x^2/2) (1-\partial_x^2)^{-1}  \Psi'
$ and that 
  \eqref{psi3} and the positivity  of the kernel associated to $(1-\partial_x^2)^{-1} $   yield 
  $$
(1-\partial_x^2) \Psi' \ge \frac{1}{2} \Psi' \Rightarrow
(1-\partial_x^2)^{-1} \Psi'\le 2  \Psi' \; 
  $$
  so that 
  $$
  \int_{\R} h  \Psi' \le 2 \int_{\R}  (u^2+u_x^2/2)\Psi'\quad  .
  $$
  Gathering these estimates we eventually obtain
  \begin{align*}
\frac{d}{dt} \Lambda_z(t) & \le   \theta  \int_{\R} (u^2+u_x^2) \Psi'(\cdot +\frac{\theta}{2}t -z)-\frac{\theta}{4} \int_{\R} u_x^2\Psi'(\cdot +\frac{\theta}{2}t -z)- \frac{\theta}{2} 
\int_{\R} h  \Psi'(\cdot +\frac{\theta}{2}t -z)\nonumber\\
& \le   \theta  \int_{\R} (u^2+u_x^2) \Psi'(\cdot +\frac{\theta}{2}t -z)-\frac{\theta}{4} \int_{\R} u_x^2\Psi'(\cdot +\frac{\theta}{2}t -z)- \theta 
\int_{\R} (u^2+u_x^2)  \Psi'(\cdot +\frac{\theta}{2}t -z)\nonumber\\
& \le 0 
\end{align*}
which yields the desired  result.
  \subsection{Asymptotic stability of well ordered trains of antipeakons-peakons}
 In \cite{EL3} the orbital stability in $ H^1(\R) $ of  well ordered trains of antipeakons-peakons is established. More precisely, the following theorem \footnote{Actually, in the  statement  given in \cite{EL3}, $ \partial_x \varphi_{c_i} $ appears instead of $ \rho_{n_0} \ast \partial_x \varphi_{c_i} $ in the orthogonality condition \eqref{mod2}. But it was noticed in \cite{L} that then  there is a gap in the proof of the  $ C^1 $-regularity of the functions $ x_i $, $ i=1,..,N$.  The proof of this version uses exactly the same arguments as developed  in the appendix of \cite{L}.}  is proved :
\begin{theorem}[\cite{EL3}]\label{mult-peaks}
Let be given $ N_- $ negative   velocities $ c_{-N_-} <..<c_{-2}<c_{-1}<0 $, $ N_+ $ positive velocities $0<c_1<c_2<..<c_{N_+} $ 
There exist   $ n_0\in\N $ satisfying \eqref{unic}, $A>0 $, $ L_0>0 $
 and $ \varepsilon_0>0 $ such that if  $ u\in C(\R_+;H^1) $ is
   the  solution of (C-H) emanating from $ u_0\in Y $, satisfying Hypothesis \ref{hyp} with 
   \begin{equation}\label{huhu}
 \|u_0-\sum_{j=-N_-\atop j\neq 0}^{N_+}  \varphi_{c_j}(\cdot-z_j^0) \|_{H^1} \le \varepsilon_0^2 
 \end{equation}
 for some  $ 0<\varepsilon<\varepsilon_0$ and $ z^0_{N_-}<z^0_{N_-+1}<\cdot\cdot\cdot<z_{N_+} $ such that 
 $$
  z_j^0-z_{j-1}^0\ge L\ge L_0 ,   \quad \forall  j\in [[-N_- +1, N_+]] \; , 
 $$
  then there exist $ N_-+N_+ $ $C^1$-functions $t\mapsto x_{-N_-}(t), .., x_{-1}(t), x_1(t), ..,t \mapsto x_{N_+}(t) $ uniquely determined such that
\begin{equation}
\sup_{t\in\R+} \|u(t,\cdot)-\sum_{j=-N_-\atop j\neq 0}^{N_+} \varphi_{c_j}(\cdot-x_j(t)) \|_{H^1} \le
A\sqrt{\sqrt{\varepsilon}+L^{-{1/8}}}\;  \label{ini2}
\end{equation}
and 
\begin{equation}
\int_{\R} \Bigl( u(t,\cdot) - \sum_{j=-N_-\atop j\neq 0}^{N_+}\varphi_{c_j}(\cdot- x_j(t)) \Bigr)
(\rho_{n_0}\ast \partial_x \varphi_{c_i}) (\cdot - x_i(t)) \, dx = 0 \; , \quad i\in\{1,..,N\}. \label{mod2}
\end{equation}
Moreover,   for $ i\in [[-N_-,N_+]]/\{0\}$, 
\begin{equation}\label{difdif}
|\dot{x}_i-c_i| \le A \sqrt{\sqrt{\varepsilon}+L^{-{1/8}}}, \quad \forall t\in\R_+  \; .
\end{equation}
 \end{theorem}
 Combining this result with the asymptotic stability of a peakon established in the preceding section, we are able to extend the asymptotic result to a train of well ordered 
 antipeakons-peakons by following the strategy developped in \cite{MMT} (see also \cite{EM}).\vspace*{2mm} \\
  {\bf Proof of Theorem \ref{asympt-mult-peaks}.}
 We first notice that it suffices to prove the result for the part of the train travelling at the right of the line $ x=\theta_0 t $, i.e. 
  there exist $ 0< c_1^*<..<c_{N_+}^*  $ and $ C^1$-functions $t\mapsto x_j(t), $,  with  $ \dot{x}_j(t) \to c_j^* $ as $ t\to +\infty $, 
 $ j\in [[1,N_+]]$,  such that 
\begin{equation}\label{mul22}
u-\-\sum_{j=1}^{N_+} \varphi_{c_j^*}(\cdot-x_j(t)) \tendsto{t\to +\infty} 0 \mbox{ in } H^1\Bigl(x>\theta_0 t\Bigr)\; .
\end{equation}
 Indeed, then the result at the left of $ x=-\theta_0 t $  follows by simply using the symmetry $ u(t,x)\mapsto -u(t,-x) $ of the C-H equation. 
 
 As in \cite{MMT} to prove the asymptotic result for the part of the train that travels to the right we proceed by induction starting from the fastest bump.
  More precisely, setting $ z_1(t) =\theta_0 t $ and 
  $$
  z_i(t)=\frac{x_{i-1}(t)+x_i(t)}{2} , \quad \forall  i\in \{2,..,N\} ,
  $$
  we will prove by induction from $ i=N$ to $ i=1 $ that $x_i(t) \to c_i^* $ as $ t\to +\infty $ and that 
  $$
  u(t)-\sum_{j=i}^N \varphi_{c_I^*}(\cdot - x_i(t)) \to 0 \quad \text{ in } H^1(]z_i(t),+\infty[ 
  $$
  Now, we notice that \eqref{huhu} and \eqref{dodo} ensure that $ x_0(t)\le x_1(t)+2 $ and thus $\supp y_-(t)\subset ]-\infty, z_2(t)-1] $. Therefore   applying  the almost monotonicity result, we obtain that the total variation of $ y(t) $ is uniformly  in time bounded on $[z_2(t),+\infty[ $ and thus the N-1 first steps of the induction can be proven exactly as
   in \cite{L} (we only have to replace $y(t) $ by $ \Theta(\cdot-z_2(\cdot)) y $ where $ \Theta $ is a continuous non decreasing function that equal $ 1 $ on 
    $[1,+\infty[$ and vanishes on $\R_-$.
    
    It remains to tackle the slowest bump travelling to the right. By separating the two same cases as in the preceding section and using the monotonicity result at the right of  $x_1(t) $    that is proven in [\cite{L}, Lemma 6.2] (Note that we can directly apply this lemma here since $\supp  y_-(t) \subset]-\infty, x_1(t)+2] $ ), we obtain in the same way  
    an asymptotic object that gives rise to an $ Y_+$ almost localized solution of \eqref{CH}.  The rest of the proof follows exactly the same line as in \cite{L} since it does not used the non negativity of  the momentum  density.
\subsection{Asymptotic stability of a  not well-ordered train of antipeakons-peakons}
Finally we notice that we can drop the hypothesis that the set of peakons and the set of anti-peakons are each one well-ordered. Note anyway that it is crucial for us that the set of antipeakons are at the left of the set of peakons since otherwise our initial datum cannot satisfy Hypothesis \ref{hyp}.
To do so, we make use of the fact that Camassa-Holm equation possesses special solutions called multipeakons given by 
 $$
 u(t,x)=\sum_{i=1}^N p_i(t) e^{-|x-q_i(t)|} 
 $$
 where $ (p_i(t),q_i(t)) $, $ i=1,..,N$,  satisfy a differential hamiltonian system (cf. \cite{CH1}). In \cite{Beals0} (see also \cite{CH1}), the limit as $ t\to \mp \infty $ of $p_i(t) $ and $ \dot{q}_i(t) $, $ i=1,..,N$, are determined. Combining the orbital stability of well ordered train of antipeakons-peakons, the continuity with respect to initial data in $ H^1(\R) $ and the asymptotics of multipeakons, the $ H^1$-stability of the variety 
defined for $ (N_-,N_+)\in (\N^*)^2 $ by  
\begin{align*}
{\mathcal N}_{N_-,N_+}:&= \Bigl\{ v=\sum_{j=-N_-\atop j\neq 0 }^{N_+} p_j e^{-|\cdot-q_j|}, \\
&(p_{-N_-},..,p_{-1},p_1,..,p_{N_+})\in (\R_-^*)^{N_-}\times (\R_+^*)^{N+} ,q_{N_-}<q_{-1}<q_1<..<q_N \,\Bigr\} \; .
\end{align*}
 is proved in (\cite{EL3}, Corollary 1.1). Gathering this last result  with the asymptotics of the multipeakons and Theorem \ref{asympt-mult-peaks}, the following asymptotic stability result for not well ordered train of peakons can be deduced quite  directly.
\begin{corollary} \label{cor-mult-peaks}
Let be given $ N_- $ negative  real numbers $ c_{-N_-} <..<c_{-2}<c_{-1}<0 $,  $ N_+ $ positive  real numbers $0<c_1<c_2<..<c_{N_+} $,  
 $N_-+N_+ $  real numbers $ q_{-N_-}^0<..<q_{-1}^0<q_1^0< ..< q_N^0 $ and let $ \lambda_{-N_-}<\cdot\cdot<\lambda_{-1}<0<\lambda_1<\cdot\cdot<\lambda_{N_+} $ be the
 $ N_-+N_+$  distinct eigenvalues of the matrix $ (p_j^0 e^{-|q_i^0-q_j^0|/2})_{(i,j)\in ([[-N_{-}, N_+]]/\{0\})^2} $. 
For any $ B> 0 $ there exists a positive function $ \varepsilon $  with $ \varepsilon(y) \to 0 $ as $ y\to 0 $ and $ \alpha_0>0 $ such that if $ u_0\in
Y$ satisfies the Hypothesis \ref{hyp} with
\begin{equation}
\|m_0\|_{\mathcal M}\le B \quad  \mbox{ and }\quad
\|u_0-\sum_{j=-N_-\atop j\neq 0}^{N_+} p_j^0 \exp (\cdot-q_j^0) \|_{H^1}\le \alpha
\label{ini3}
\end{equation}
for some $ 0<\alpha<\alpha_0 $, 
 then there exists $c_{-N_-}^*<\cdot\cdot<c_{-1}^*<0<c_1^*<\cdot\cdot <c_{N_+}^*  $  and $ C^1$-functions $ (x_{-N_-},..,x_{-1},x_1,..,x_{N_+}) $ with 
 $$
 |c_i^*-\lambda_i|\le \varepsilon(\alpha) \quad \text{and} \quad \lim_{t\to +\infty} \dot{x}_i(t) = c_i^* \; ,\quad \forall i\in [[-N_{-}, N_+]]/\{0\},
 $$
    such that 
\begin{equation}\label{coromul2}
u-\sum_{j=-N_-\atop j\neq 0 }^{N_+}  \varphi_{c_i^*}(\cdot-x_i(t)) \tendsto{t\to +\infty} 0 \mbox{ in } H^1(|x|>\min(-\lambda_{-1},\lambda_1)/4)\; .
\end{equation}
\end{corollary}

    \section{Appendix: An improvement of the asymptotic stability result in the class of solutions with non negative momentum density} 
 In this subsection we make a simple observation that enables to improve the asymptotic stability result given in \cite{L}.
 This observation is that all  the energy of  the solutions of the C-H equation, that have a non negative density momentum,
   is traveling to the right. Moreover, as stated in the following lemma,  the energy at the left of any fixed point tends to zero as $ t\to +\infty $.
   \begin{lemma}\label{apen}
   For any $ u_0\in Y_+ $ and any $ z\in\R  $, denoting by $ u\in C(\R;H^1) $  solution of \eqref{CH} emanating from $ u_0$  it holds 
 \begin{equation}\label{tdt}
\lim_{t\to+\infty} \|u(t)\|_{H^1(]-\infty, z[)} = 0 \; .
 \end{equation}
 \end{lemma}
 This differs from KdV like equations since for these last equations the linear waves are travelling to the left. This causes a small loss of energy to the left and the so called dispersive tail. Note that for CH-type equation, there is no linear part and if moreover  the momentum density is non negative, there is also no (even small)  antipeakon and thus all the energy is travelling to the right that  formally leads to \eqref{tdt}.\\
 Before proving Lemma \ref{tdt}, let us state the improved asymptotic stability result : 
   \begin{theorem} \label{asympstab+} 
Let $ c>0 $ be fixed. There exists a  universal constant $0<\eta_0\ll 1  $ such that for any $ 0<\theta<c $ and any $ u_0\in Y_+ $ satisfying 
\begin{equation}\label{difini+}
\|u_0-\varphi_c \|_{H^1} \le \eta_0 \Bigl(\frac{\theta}{c}\Bigr)^{8}\; ,
\end{equation}
there exists $ c^*>0 $ with $ |c-c^*|\ll c $ and a $C^1$-function $ x \, : \, \R \to \R $ 
 with $ \displaystyle\lim_{t\to \infty} \dot{x}=c^* $  such that
\begin{equation}
u(t,\cdot+x(t)) \weaktendsto{t\to +\infty} \varphi_{c^*} \mbox{ in } H^1(\R) \; ,
\end{equation}
where $ u\in C(\R; H^1) $ is the solution emanating from $ u_0 $.\\
Moreover, for any  $ z>0 $, 
\begin{equation}\label{cvforte+}
\lim_{t\to +\infty} \|u(t) -\varphi_{c^*}(\cdot-x(t))\|_{H^1(\R/[z, \theta t])}=0 \; .
\end{equation}
\end{theorem}

 \noindent
 {\bf Proof of Lemma \ref{apen}.}
 Let $ 0<\gamma<\|u_0\|_{H^1}^2$ and let $ x_\gamma \, :\, \R\to \R $ be defined by 
 \begin{equation}\label{defxg}
 \int_{\R} (u^2+u_x^2)(t) \Phi(\cdot -x_\gamma(t)) =\gamma 
 \end{equation}
 with $ \Phi $  defined in \eqref{defPhi}.
 Note that $ x_\gamma(\cdot) $ is well-defined  since $ u_0\in Y_+ $ forces $ u>0 $ on $ \R^2 $ and thus for any fixed $ t\in \R $, 
 $ z\mapsto \int_{\R} (u^2+u_x^2)(t) \Phi(\cdot-z) $ is a decreasing continuous bijection from $ \R $ to $ ]0,\|u_0\|_{H^1}^2[ $. Moreover, $ u\in C(\R; H^1) $ ensures that $ x_\gamma(\cdot) $ is a continuous function.  \eqref{tdt} is clearly a direct consequence of the fact that 
  \begin{equation}\label{td3}
 \lim_{t\to +\infty} x_\gamma (t)=+\infty \; .
 \end{equation}
 To prove  \eqref{td3} we first claim that for any $ t\in \R $ and any $ \Delta>0 $ it holds 
 \begin{equation}\label{td4}
 x_\gamma(t+\Delta)-x_\gamma(t) \ge \frac{1}{2}\Bigl(\int_{t}^{t+\Delta} \int_{x_\gamma(t)}^{x_\gamma(t)+2} u^2(\tau,s) \, ds\, d\tau  \Bigr)^{1/2}>0 \; .
 \end{equation}
 Let us prove this claim. First we notice that by continuity with respect to initial data, it suffices to prove \eqref{td4} for $ u\in C^\infty(\R; H^\infty)
 \cap L^\infty(\R; Y_+)  $. Then a simple application of the implicit function theorem ensures that $ t\mapsto x_\gamma (t) $ is of class $ C^1 $. Indeed, 
 $$
 \psi \, :\, (z,v)\mapsto \int_{\R} (v^2+v_x^2) \Phi(\cdot-z) 
 $$
 is of class $ C^1 $ from $ \R \times H^1(\R) $ into $ \R $ and for any $ (z_0,v)\in \R\times Y_+/\{0\} $, 
 $\partial_z \psi(z_0,v)=  \int_{\R}  (v^2+v_x^2)\Phi'(\cdot-z_0)>0 $.   It then follows from \eqref{go} that
 $$
   \dot{x}_\gamma(t) \int_{\R} (u^2 +u_x^2)\Phi'(\cdot-x_\gamma(t))= \int_{\R} u u_x^2 \Phi'(\cdot-x_\gamma(t)) +2 \int_{\R} u h  \Phi'(\cdot-x_\gamma(t))
$$
 where $ h=p\ast (u^2+u_x^2/2) $. Now, it is worth noticing that according to \cite{C}, the following convolution estimate holds : For any $ v\in H^1(\R) $, 
   \begin{equation}\label{td5}
 p\ast (v^2+v_x^2/2) (x)\ge \frac{1}{2} v^2 \quad \text{ on } \R.
 \end{equation}
 Combining \eqref{td4}-\eqref{td5} and the fact that $ |v_x|\le v $ on $ \R $ for any $ v\in Y_+ $,  we eventually get 
 $$
2\dot{x}_\gamma (t) \int_{\R} u^2 \Phi'(\cdot-x_\gamma(t))  \ge \int_{\R} u(u^2+u_x^2) \Phi'(\cdot-x_\gamma(t))\ge
  \int_{\R} u^3 \Phi'(\cdot-x_\gamma(t)) 
  $$
and H\"older inequality together with the fact that $\Phi' $ is a non negative function of total mass 1 lead to
 \begin{equation}\label{td6}
 \dot{x}_\gamma(t) \ge \frac{1}{2}\Bigl( \int_{\R} u^2 \Phi'(\cdot-x_\gamma(t))\Bigr)^{1/2}
 \end{equation}
 Integrating this inequality between $ t $ and $ t+\Delta $  yields  \eqref{td4} that obviously implies that $x_\gamma(\cdot) $ is an increasing function. In particular there exists $ x_\gamma^\infty\in \R\cap \{+\infty\} $ such that $x_\gamma(t) \nearrow x_\gamma^\infty $ as $ t\to +\infty $ and it remains to prove that  $x_\gamma^\infty=+\infty $. Assuming the contrary, we first notice that since $ |u_x|\le u \le \|u_0\|_{H^1} $ on $ \R^2 $, \eqref{defxg} then yields 
 \begin{equation}\label{contro}
 \lim_{t\to +\infty}  \int_{\R} (u^2 +u_x^2)\Phi'(\cdot-x_\gamma(t))= \lim_{t\to +\infty}  \int_{\R} (u^2 +u_x^2)\Phi'(\cdot-x_\gamma^\infty)=\gamma \; .
 \end{equation}
 Now, taking $ \Delta=1 $, \eqref{td4} forces
 $$
 \lim_{t\to  +\infty} \int_{t}^{t+1} \int_{x_\gamma(t)}^{x_\gamma(t)+2} u^2(\tau,s)\, ds\,d\tau =0 
 $$
 which, recalling that $ x_\gamma(t)\to  x_\gamma^\infty$ , leads to 
  $$
 \lim_{t\to  +\infty} \int_{t}^{t+1} \int_{x_\gamma^\infty}^{x_\gamma^\infty+2}u^2(\tau,s) \, ds\, d\tau =0 \; .
 $$
 In particular there exists a sequence $(t_n,x_n)_{n\ge 1}  \subset \R\times  [x_\gamma^\infty,x_\gamma^\infty+2] $ with $ t_n\nearrow +\infty $ such that 
 $ u(t_n,x_n) \to 0 $ as $ n\to \infty $. Therefore,  making use of the fact that $ |u_x|\le u $ on $ \R^2 $ forces, for any $(t,x_0)\in \R^2 $,  that 
 \begin{equation}
 u(t,x) \le e^{|x_0-x|} u(t,x_0) , \quad \forall x\in \R \;  ,
 \end{equation}
 we infer that for any $ A>0 $,
 \begin{equation}\label{td7}
 \lim_{n\to \infty} \sup_{ x\in[x_\gamma^\infty-A,  x_\gamma^\infty+A]} u(t_n,x) =0 \; .
 \end{equation}
Finally, taking $ A>0 $ such that $ x_\gamma^\infty-A<x_{\gamma'}(0) $ with $\gamma < \gamma'<\|u_0\|_{H^1}^2 $, we infer from \eqref{td7} and the monotony of 
 $  t\mapsto x_{\gamma'}(t) $ that 
 $$
 \lim_{n\to \infty}\int_{\R} (u^2+u_x^2)(t_n,\cdot) \Phi(\cdot - x_\gamma^\infty) =\gamma' \; .
 $$
This contradicts \eqref{contro} and concludes the proof of the lemma.$\hfill \square$\vspace*{2mm} \\
  \noindent
 {\bf Conflict of Interest }: The author declares that he has no conflict of interest.

\end{document}